\documentclass[12pt]{amsart}
\usepackage{graphicx}
\usepackage{amssymb}
\usepackage{amsfonts}
\usepackage{hyperref}
\usepackage{cancel}

\usepackage[normalem]{ulem}

\usepackage[margin=2cm]{geometry}
\usepackage{mathrsfs}
\swapnumbers
\usepackage[usenames,dvipsnames,svgnames,table]{xcolor}

  [1995/08/01 v0.1 Double stroke roman fonts]

\def\ds@whichfont{dsrom}
\relax

\DeclareMathAlphabet{\mathds}{U}{\ds@whichfont}{m}{n}

\numberwithin{equation}{section}
\newtheorem{theorem}{Theorem}
\newtheorem{lemma}[theorem]{Lemma}
\newtheorem{corollary}[theorem]{Corollary}

\newtheorem{proposition}[theorem]{Proposition}
\theoremstyle{definition}
\newtheorem{definition}[theorem]{Definition}
\newtheorem{assumption}[theorem]{Assumption}
\newtheorem{remark}[theorem]{Remark}
\newtheorem{example}[theorem]{Example}

\theoremstyle{plain}

\numberwithin{figure}{section} 
\theoremstyle{plain}
\theoremstyle{plain}
\theoremstyle{remark}
\newtheorem*{acknowledgement*}{Acknowledgement}
\theoremstyle{example}

\newcommand{\cA}{{\mathcal A}}
\newcommand{\cB}{{\mathcal B}}

\newcommand{\cD}{{\mathcal D}}
\newcommand{\cE}{{\mathcal E}}
\newcommand{\cF}{{\mathcal F}}
\newcommand{\cG}{{\mathcal G}}
\newcommand{\cH}{{\mathcal H}}

\newcommand{\cJ}{{\mathcal J}}

\newcommand{\cL}{{\mathcal L}}

\newcommand{\cX}{{\mathcal X}}

\newcommand{\te}{{\theta}}

\newcommand{\Om}{{\Omega}}
\newcommand{\om}{{\omega}}
\newcommand{\ve}{{\varepsilon}}
\newcommand{\del}{{\delta}}

\newcommand{\sig}{{\sigma}}
\newcommand{\al}{{\alpha}}
\newcommand{\be}{{\beta}}

\newcommand{\la}{{\lambda}}

\newcommand{\bbC}{{\mathbb C}}
\newcommand{\bbE}{{\mathbb E}}

\newcommand{\bbN}{{\mathbb N}}
\newcommand{\bbP}{{\mathbb P}}
\newcommand{\bbR}{{\mathbb R}}

\newcommand{\bbI}{{\mathbb I}}

\def\fm{{\mathfrak{m}}}

\newcommand{\DS}{\displaystyle}

\begin{document}
\title[]{Non-uniform Edgeworth expansions for weakly dependent random variables and their applications}
 \author{Yeor Hafouta}
\address{The University of Florida}

\email{yeor.hafouta@mail.huji.ac.il}

\dedicatory{  }
 \date{\today}
\maketitle

\begin{abstract}\noindent
We obtain non-uniform Edgeworth expansions for several classes of weakly dependent (non-stationary) sequences of random variables, including uniformly elliptic inhomogeneous Markov chains, random and time-varying (partially) hyperbolic or expanding dynamical systems, products of random matrices and some classes of local statistics. To the best of our knowledge this is the first time such results are obtained beyond the case of independent summands, even for stationary sequences.
As an application of the non uniform expansions we obtain average versions of Edgeworth exapnsions, which provide estimates of the underlying distribution function in $L^p(dx)$ by the standard normal distribution function and its higher order corrections.
An additional  application is to expansions of expectations $\bbE[h(S_n)]$ of  functions $h$   of the underlying sequence $S_n$, whose derivatives grow at most polynomially fast. In particular we provide expansions of the moments of $S_n$ by means the variance of  $S_n$. A third application is to Edgeworth expansions in the Wasserstein distance (transport distance). In particular we prove Berry-Esseen theorems in the Wasserstein metrics. This paper compliments \cite{NonU BE} where non-uniform Berry-Esseen theorems were obtained.


 \end{abstract}
 

\section{Introduction}

\subsection{Berry-Esseen theorems and Edgeworth expansions} 
Let $S_n$ be a sequence of centered random variables so that $\sig_n=\|S_n\|_{L^2}\to\infty$. Recall that $S_n$ obeys the (self-normalized) central limit theorem (CLT) if $W_n=S_n/\sig_n$ converges in distribution to the standard normal law, namely for every $x\in\bbR$,
$$
\lim_{n\to\infty}\bbP(W_n\leq x)=\Phi(x):=\frac{1}{\sqrt{2\pi}}\int_{-\infty}^x e^{-t^2/2}dt.
$$ 
In this paper we will consider partial sums $S_n=\sum_{j=1}^nX_j$. In the past decades  CLT have been proved to many classes of weakly dependent  summands $X_j$. Recall that, in general, the optimal rate is $O(\sigma_n^{-1})$, namely
we say that $W_n$ obeys the  Berry Esseen theorem  if 
$$
\sup_{x\in\bbR}\left|\bbP(W_n\leq x)-\Phi(x)\right|=O(\sig_n^{-1}).
$$ 

First results for independent summands were proven in \cite{Berry, Ess}, and  by now optimal rates have been established for several classes of weakly dependent summands
when $\sig_n^2=\text{Var}(S_n)$ grows linearly fast in $n$ (see \cite{Nag65, Stein72, CJ, RE, GH, SaulStat, RioBE, HH, GO, Jir0, HK} for a partial list). Three exceptions are  \cite{DolgHaf, DolgHaf PTRF 2, MarShif1} where optimal rates are obtained for additive functionals of  inhomogeneous Markov chains, for sequential dynamical systems and for Markov shifts. In these cases  the variance $\sig_n^2$ of the underlying partial sums $S_n$ can grow arbitrary slow. 

A refinement of the  Berry Esseen theorem is the, so-called, Edgeworth expansion. We say that $W_n$ obeys the (uniform) Edgeworth expansion of order $r\in\bbN$ if there are polynomials $H_{j,n}$ with bounded coefficients whose degrees do not depend on $n$ so that 
$$
\sup_{x\in\bbR}\left|\bbP(W_n\leq x)-\Phi(x)-\varphi(x)\sum_{j=1}^{r}\sigma_n^{-j}H_{j,n}(x)\right|=o(\sigma_n^{-r})
$$ 
where $\varphi(x)=\frac{1}{\sqrt{2\pi}}e^{-x^2/2}$.  
When taking $r=1$ the corresponding (first order) Edgeworth expansion (when it holds\footnote{An immediate obstruction for the first order expansions is the case when $S_n$ take values on some lattice, since then the distribution function of $W_n$ has jump of order $1/\sig_n$.}) reveals exactly when better than $O(\sig_n^{-1})$ CLT rates are achieved. In applications  (see \cite{Feller, IL, Bobkov2016, DolgHaf}) the coefficients of  $H_{r,n}$ are defined by means of the first $r+2$ moments of $S_n$, and for $r=1$ they are defined by means of $\bbE[S_n^3]/\sig_n^2$. Thus, under the validity of the first order expansions we get better than the optimal $O(\sig_n^{-1})$ rates if and only if $\bbE[S_n^3]=o(\sig_n^2)$, and so  the second obstruction for better than optimal rates comes from the magnitude of the third moment.
When $S_n$ is a partial sum of a stationary sequence $X_j$ or when it satisfies a certain ``weak form of stationarity" (see Section \ref{WeakForm}), then it is natural also to consider only ``stationary expansions" and require that $H_{j,n}=H_j$ does not depend on $n$. In this case the polynomials $H_j$ must be  unique (see, for instance, \cite{FL}).

Expanding distribution functions goes back to \cite{Cheb},  \cite{EdgeOr} and \cite{Cr28}. 
Since then Edgeworth expansions were obtained by many authors in many different setups.
For independent and identically distrubuted random variables it was proven by Esseen in \cite{Ess} 
that the expansion of order 1 holds iff
the distribution of $S_N$ is non-lattice. Therefore, better than optimal rates are obtained if and only if the distribution is non-lattice and $\bbE[S_n^3]=0$ (e.g. when the random variables are symmetric).

The conditions for higher order expansions are not yet completely understood.
Sufficient conditions for the Edgeworth expansions of an arbitrary order
 were first obtained in \cite{Cr28} under the assumption that the characteristic
function 
of the sum $\bbE(e^{itS_N})$ 
decays exponentially fast as $N\to\infty$ uniformly for large $t$ (this is the, so called, Cramer condition).
Later on, the same expansions were obtained in \cite{Ess, Feller, BR76, Br, AP}
under 
weaker decay conditions\footnote{The decay conditions used in the above papers are optimal, since one can provide examples where the decay is slightly weaker
and there are oscillatory corrections to Edgeworth expansion, see  \cite{DF, DH, DolgHaf}.},
where the second  paper
considered non identically distributed variables and  the fourth and fifth  considered random iid vectors. 
 Edgeworth expansions were also proven for several classes
of weakly dependent random variables including 
stationary Markov chains (\cite{Nag1, Nag2, FL}), chaotic dynamical systems 
(\cite{CP, FL, FP}) and certain classes of local statistics
(\cite{BGVZ1,Hall, BGVZ2, CJV}).
In particular,  Herv\'e-P\`ene proved in
\cite{HP}  that for several classes of stationary processes the first order 
Edgeworth expansion holds true if $S_N$ is irreducible, in the sense that $S_N$ can not be
represented as $S_N'+H_N$ where $S_N'$ is lattice valued and $H_N$ is bounded.
Let us also mention that
in \cite{Bar, RinRot}  certain  weak expansions were obtained, i.e.
expansions of the form $\bbE\left(\phi\left(S_N/\sig_N\right)\right)$ where $\phi$ is a smooth
test function. Finally, in \cite{DolgHaf} we obtained optimal conditions for Edgeworth expansions for additive functionals of uniformly elliptic inhomogeneous Markov chains (without assumptions on the growth rates of $\sig_n$). This seems to be the only result in literature where expansions are obtained without growth assumptions on $\sig_n$.

\subsection{Main results: non-uniform Edgeworth expansions}\label{Intro2}
As opposed to the results discussed in the previous section, this paper concerning non-uniform estimates, which for the sake of convenience are described here as a definition. Let us consider a sequence of the form $W_n=\frac{S_n-A_n}{B_n}$ where $A_n=O(1)$, $B_n=\sig_n+O(1)$, $\sig_n=\|S_n\|_{L^2}$ and $S_n\in L^2$ has zero mean.
\begin{definition}[Non Uniform Exapnsions]
We say that $W_n$ obeys the (non-stationary) Edgeworth expansion of order $r$ with power $s\geq 0$ if for every $x\in\bbR$,
$$
\cD_{r,n}(x):=\left|\bbP(W_n\leq x)-\Phi(x)-\varphi(x)\sum_{j=1}^{r}B_n^{-j}H_{j,n}(x)\right|\leq C_n(1+|x|)^{-s}B_n^{-r}
$$
where $C_n\to 0$
and  $H_{j,n}$ are polynomials with uniformly bounded coefficients and degrees not depending on $n$.
\end{definition}

\begin{definition}[Stationary Non Uniform Exapnsions]
We say that $W_n$ obeys the \textit{stationary} Edgeworth expansion of order $r$ with power $s\geq 0$ if for every $x\in\bbR$,
$$
\cD_{r,n}(x):=\left|\bbP(W_n\leq x)-\Phi(x)-\varphi(x)\sum_{j=1}^{r}B_n^{-j}H_{j}(x)\right|\leq C_n(1+|x|)^{-s}B_n^{-r}
$$
where $C_n\to 0$
and  $H_{j}$ are polynomials not depending on $n$.
\end{definition}

\begin{remark}
    The case $s=0$ corresponds to the usual Edgeworth expansions, which from now on will be referred to as the uniform Edgeworth expansions. When $s>0$ we will refer to the expansions as non-uniform ones.
\end{remark}

For partial sums of iid summands non-uniform Edgeworth expansions were obtained under various conditions, see for instance \cite{Nag65}, \cite{Osipov67},  \cite{Osipov72} and \cite{Petrov72}.  However, despite the variety of applications of non uniform Edgeworth expansions (described  in the following sections), the literature about non-uniform estimates for weakly dependent random variables is not as vast as the literature in the uniform case. The goal of this paper is to extend the known results from the uniform case to the non-uniform case, as will be discussed in the following sections. It seems that these are the first results in the direction beyond the case of independent summands.
 This will be done by using an abstract scheme that allows us to capture
 the following types of sequences:
\begin{itemize}
\item Partial sums $S_n=\sum_{j=0}^{n-1}g\circ T^j$ generated by a chaotic  dynamical system $T$ and a sufficiently regular observable $g$ (c.f. \cite{FL, FP} for stationary uniform expansions);
\vskip0.1cm

\item Partial sums $S_n=\sum_{j=1}^n g(\xi_j)$ generated by an homogeneous geometrically ergodic\footnote{Instead of geometric ergodicity, we can assume that the corresponding Markov operator has a spectral gap on an appropriate Banach space $B$, and that $g\in B$.} Markov chain $\{\xi_j\}$ and a bounded function $g$ (c.f. \cite{FL} for stationary uniform exapansions).
\vskip0.1cm
\item  partial sums $S_n=\sum_{j=1}^{n}f_j(\xi_j,\xi_{j+1})$ of additive functionals $f_j$ of uniformly elliptic (not necessarily homogeneous) Markov chains $\xi_j$ (c.f. \cite{DolgHaf} for non-stationary uniform expansions).
\vskip0.1cm
\item Partial sums generated by a random  (sequential) choatic dynamical systems and a random (sequential) sufficiently regular observable (c.f. \cite{Ha});
\vskip0.1cm
\item $S_n=\ln \|A_nA_{n-1}\cdots A_1x\|$ where $A_j$ are iid (strongly irreducible) random matrices and $x$ is a unit vector (c.f. \cite{FP} for uniform expansions);
\vskip0.1cm
\item Some classes of local statistics (c.f. \cite[Sections 3-5]{Dor} for uniform expansions);
\end{itemize}
As will be discussed in the next sections, once the appropriate non-uniform estimates are obtained several other results will follow (see also Remark \ref{h rem}). As noted in the brackets above, in all the above examples uniform Edgeworth expansions are known under suitable conditions. In this paper we will focus on the non-uniform case.

\subsection{Applications of the non-uniform estimates}\label{IntApp1}

\subsubsection{\textbf{$L^p$ type Gaussian estimates of distribution function}}
The non-uniform 
Edgeworth expansions of order $r$ and power $s$ yield that for all $p>1/s$ we have
$$
\|\cD_{r,n}\|_{L^p(dx)}=o(B_n^{-r}).
$$
Such results where obtained for partial sums of independent summands in \cite{[3]}, and here we consider the weakly dependent case.

\subsubsection{\textbf{Expansions of functions of $W_n$ with polynomially fast growing derivatives}}
Let $h:\bbR\to\bbR$ be  an a.e. differentiable function so that   $H_s=\int\frac{|h'(x)|}{(1+|x|)^s}dx<\infty$. Then, with $F_n(x)=\bbP(W_n\leq x)$ we have
\begin{equation}\label{int form}
\bbE[h(W_n)]=-\bbE\left[\int_{W_n}^{\infty}h'(x)dx\right]=-\int_{-\infty}^{\infty}h'(x)\bbP(W_n\leq x)dx=
-\int_{-\infty}^{\infty}h'(x)F_n(x)dx
\end{equation}
and so 
$$
\left|\bbE[h(W_n)]-\int h(x)d\cD_{r,n}(x)\right|\leq\int |h'(x)|\cD_{r,n}(x)dx\leq CH_sC_nB_n^{-r},\, C_n\to0
$$
where we set $\cD_{0,n}(x)=\Phi(x)$.
In particular we obtain expansions for the moments of $W_n$ by taking $h(x)=x^p$ for all $p<s$. 
Compared with existing results
in \cite{Bar, RinRot}  similar expansions were obtained for smooth 
test functions $h$. Here we can consider less regular functions and arbitrary growth rates of the variance.

\subsubsection{\textbf{Transport distances}}\label{Intro3}
Given two Borel probability measures $\mu$ and $\nu$ on the real line $\bbR$ with absolute moment of order $p$, we set
$$
W_p(\mu,\nu)=\inf_{\pi}\left(\int_{-\infty}^\infty\int_{-\infty}^\infty|x-y|^pd\pi(x,y)\right)^{\frac1p}
$$
where the infimum is taken over all the probability measures on $\bbR\times\bbR$ with marginals $\mu$ and $\nu$. Namely
$$
W_p(\mu,\nu)=\inf_{(X,Y)}\left\|X-Y\right\|_{L^p}
$$
where the infimum is taken over the class of random variables $(X,Y)$ on $\bbR\times\bbR$ so that $X$ is distributed according to $\mu$ and $Y$ is distributed according to $\nu$. 
The function $(\mu,\nu)\to W_p(\mu,\nu)$ defines a metric on the set of Borel probability measure  with absolute moment of order $p$. 
Recently Bobkov \cite{Bobkov2018} extended $W_p$  to the class of finite signed measures with finite absolute moment of order $p$ according to the formula
$$
\tilde W_p(\mu,\nu)=\sup_{u\in U_q}\int_{-\infty}^{\infty}\left|u(F_\mu(x))-u(F_\nu(x))\right|dx
$$
where $q$ is the conjugate exponent of $p$, $U_q$ is the class of all differentiable functions $u$ so that $\|u'\|_{L^q(dx)}\leq1$
and $F_\mu(x)=\mu((-\infty,x])$ and $F_\nu=\nu((-\infty,x])$ are the, so-called, generalized distribution functions.
Henceforth we will abuse the notation and will not distinguish between $W_p$ and $\tilde W_p$.

Let $Y_1,Y_2,...,Y_n$ be centered independent square integrable random variables and set $S_n Y=\sum_{j=1}^n Y_j$. Let $\mu_n$ be the law of $S_n Y$ and let $\gamma=d\Phi$ be the standard normal law. Then in \cite{Bobkov2018} is was shown (in particular) that for every $p\geq1$, if $Y_i\in L^{p+2}$ then we have the following optimal CLT rate in the $p$-th transport distances:
$$
W_p(\mu_n,\gamma)\leq c_p L_{p+2,n}^{1/p}
$$
where $L_{s,n}=\sum_{j=1}^{n}\bbE[|Y_j|^s]$ is the $s$-th Lyapunov coefficient.
This proved the validity of a conjecture of  Rio \cite{Rio2009}.
 In the self-normalized case when $Y_j$ has the form $Y_j=Y_{j,n}=X_j/\|S_n\|_{L^2}$, with $S_n=\sum_{j=1}^{n}X_j$, if $\sum_{j=1}^n\bbE[|X_j|^{p+2}]\leq C_p\sum_{j=1}^n\bbE[|X_j|^2]$ (e.g. when $\sup_j\|X_j\|_{L^\infty}<\infty$) the above result reads 
 $$
W_p(\mu_n,\gamma)=O(\|S_n\|_{L^2}^{-1})
$$
which, in general, is the best possible CLT rate (i.e. the optimal rate).
Note that in the iid setup, when $X_1\in L^{p+2}$, both error terms $L_{p+2,n}^{1/p}$ and $O(\|S_n\|_{L^2}^{-1})$ yield the classical optimal rates
$$
W_p(\mu_n,\gamma)=O(n^{-1/2}).
$$
Remark that showing that $W_p(\mu_n,\gamma)=O(n^{-1/2})$  starting from the work of Esseen (1958) in the case $p=1$, and we refer to the introduction of \cite{Rio2009} for more details about the history of the problem and for partial solutions before the work of Bobkov.

A natural question is to get the $W_p$-CLT rates $O(\|S_n\|_{L^2}^{-1})$ for some classes of weakly dependent random variables.
In \cite[Corollary 3.2]{Bobkov2018} Bobkov showed that for every finite measures $\mu,\,\nu$ on $\bbR$ so that 
$$
\int_{-\infty}^{\infty}|x|^pd\mu(x)+\int_{-\infty}^{\infty}|x|^pd\nu(x)<\infty
$$ 
we have
\begin{equation}\label{Relation0}
W_p(\mu, \nu)\leq \int_{-\infty}^{\infty}|F(x)-G(x)|^{1/p}dx
\end{equation}
where $F(x)=\mu((-\infty,x])$ and $G(x)=\nu((-\infty,x])$ are the generalized distribution functions (this  is meaningful only when $F(\infty)=G(\infty)$). 
Applying \eqref{Relation0} we will be able to get optimal CLT rates and Edgeworth expansions in $W_p$ using our non-uniform estimates. This generalizes the main results in \cite{Bobkov2018} to inhomogeneous Markov chains and the other processes mentioned before.

We would also like to mention the recent paper \cite{Aus}, in which a similar task was obtained for some classes of stationary $\al$-mixing random fields (in which $\|S_n\|_{L^2}^2$ grows linearly fast in $n$) using different methods (an appropriate adaptation of Stein's method). While $\al$-mixing random variables include many important examples in statistics and probability most the examples we have in mind do not seem to be $\al$-mixing. Moreover, we do not require stationary of the underlying sequence, and we also obtain the more general results described in the previous sections.

\section{Main abstract result}\label{Main}
Let $(S_n)$ be a sequence of  centered random variables so that, $\sig_n=\|S_n\|_{L^2}\to\infty$. Let us define 
$W_n=\frac{S_n-A_n}{B_n}$, where  $A_n$ is a bounded sequence and $B_n=\sig_n+O(1)$.
For instance, we can take $A_n=0$ and $B_n=\sig_n$ which corresponds to the centralized self-normalized case, but we will see in some of our application to stationary sequences and products of random matrices that the choice of $A_n=c+O(\del^n)$, for some $\del\in(0,1)$, $c
\in\bbR$ and $B_n=\sig\sqrt n$, $\sig>0$ will also be natural.
Let $$F_n(x)=\bbP(W_n\leq x)$$ be the distribution function of $W_n$. In this section we will describe our non uniform Berry-Esseen theorems and Edgeworth expansions for $F_n$, under conditions which involve the difference between the logarithmic characteristic function of $S_n/\sig_n$ and the standard normal one, 
which is given by
$$
\Lambda_n(t)=\ln\bbE[e^{itS_n/\sig_n}]+t^2/2.
$$

We consider here the following two assumptions.
We first recall the following assumption from \cite{DolgHaf}.
\begin{assumption}\label{GrowAssum}
For some $m\geq2$, for all $3\leq j\leq m+1$ there exist constants $C_j,\ve_j>0$ so that
\begin{equation}\label{GAs}
\sup_{t\in[-\ve_j\sig_n,\ve_j\sig_n]}|\Lambda_n^{(j)}(t)|\leq C_j\sig_n^{-(j-2)}.
\end{equation}
\end{assumption}
\begin{remark}
Assumption \ref{GrowAssum} only concerns the derivatives of $\ln\bbE[e^{itS_n/\sig_n}]$ (since for $j\geq 3$ the second term vanishes). However,  it is more natural to present the results using the difference $\Lambda_{n}(t)$ since it measures the deviation from normality in an appropriate sense.
\end{remark}
In \cite{NonU BE} we showed that Assumption \ref{GrowAssum} is sufficient for non-uniform Berry-Esseen theorems. 
To obtain Edgeworth expansions we also consider the following assumption.
Let $f_n(t)=\bbE[e^{it S_n/\sig_n}]$ and let $f_n^{(m)}(t)$ be the $m$-th derivative of $f_n$.
\begin{assumption}\label{DerAss}
For some $m\geq3$, for  every $c>0$ (small enough) and $B>0$ (large enough) we have
$$
\int_{c\sig_n\leq |t|\leq B\sig_n^{m-2}}\left|\frac{f_n^{(m)}(t)}{t}\right|dt=o(\sig_n^{-(m-2)}).
$$
\end{assumption}

\begin{remark}\label{Rem psi}
Let $\psi_n(t)=\bbE[e^{it S_n}]$. Then $f_n(t)=\bbE[e^{it S_n/\sig_n}]=\psi_n(t/\sig_n)$, and so  $f_n^{(m)}(t)=\sig_n^{-m}\psi_n^{(m)}(t/\sig_n)$. Hence, Assumption \ref{DerAss} means that for every $c>0$ and $B>0$ large enough  we have
\begin{equation}\label{psi cond}
\int_{c\leq |t|\leq  B\sig_n^{m-3}}\left|\frac{\psi_n^{(m)}(t)}{t}\right|dt=o(\sig_n^2).
\end{equation}
\end{remark}
Our main result is:

\begin{theorem}\label{Edge2}
Let Assumptions \ref{GrowAssum} and \ref{DerAss} hold with the same $m$. Then 
 there are polynomials $H_{j,n}$ with bounded coefficients whose degrees do not depend on $n$ so that with 
\begin{equation}\label{deff}
\Psi_{r,n}(x)=\Phi(x)+\varphi(x)\sum_{j=1}^{r}B_n^{-j}H_{j,n}(x)
\end{equation}
for all $r=1,2,...,m-2$ and $x\in\bbR$ we have\footnote{Namely, the non-uniform Edgeworth expansion or order $r$ and power $m$ holds true.}
$$
\left|F_n(x)-\Psi_{r,n}(x)\right|\leq C_nB_n^{-r}(1+|x|)^{-m}, \,C_n\to 0.
$$
In particular, for all $p>1/m$, 
$$
\left\|F_n-\Psi_{k,n}\right\|_{L^p(dx)}=o(\sig_n^{-k}).
$$
The coefficients of the polynomials are polynomial functions of $A_n$, $B_n-\sig_n$ and $\Lambda_n^{(j)}(0)\sig_n^{(j-2)}$ for $j\geq3$.
\end{theorem}

Our next result is a Berry-Esseen theorem in transport distances:
\begin{theorem}\label{BE1}
Under Assumption \ref{GrowAssum}   the Berry-Esseen theorem   in the transport distance of any order $p<m-1$ holds true. Namely, 
there exists a constant $C>0$ so that 
$$
W_{p}(dF_n,d\Phi)\leq CB_n^{-1}. 
$$  
\end{theorem}

Next, as explained in Section \ref{Intro3},
the following result follows from Theorem \ref{Edge2} together with \cite[Corollary 3.2]{Bobkov2018}:

\begin{theorem}\label{Edge1}
Under Assumptions \ref{GrowAssum} and \ref{DerAss} we have the following: let $H_{r,n}$ be the polynomials from Theorem \ref{Edge2}. Then
for all $r=1,2,...,m-2$ and $p<m-1$ we have
$$
W_p(dF_n,d\Psi_{r,n})=o(B_n^{-\frac{r}{p}})
$$
where $\Psi_{r,n}$ is defined in \eqref{deff}.
\end{theorem}
As described in Section \ref{Intro3}, Theorem \ref{BE1} follows from Theorem \ref{Edge1}, together a certain smoothing argument which ``forces" Assumption \ref{DerAss} to hold. This is the reason why Theorem \ref{BE1} holds only under Assumption \ref{GrowAssum}.

\subsection{An explicit formula for the polynomials in the general self-normalized case}
Let us denote by 
$$\gamma_j(W)=i^{-j}\left(\frac{d^j}{dt^j}\,\bbE[e^{it W}]\right)\Big|_{t=0}$$
 the $j$-th cumulant of a random variable $W$ with  finite absolute $j$ moments.
Let $H_s$ be the $s$-th Hermite polynomials.
  Then in Section \ref{Form} we will see that
in the self normalized case when $A_n=0$ and $B_n=\sig_n=\|S_n\|_{L^2}$ our expansions hold true with $\Psi_{k,n}=\Phi_{k+2,n}$, where for all $m\geq3$, 
$$
\Phi_{m,n}(x)=\Phi(x)-\varphi(x)\sum_{\bar k}\frac1{k_1!\cdots k_{m-2}!}\left(\frac{\gamma_3(W_n)}{3!}\right)^{k_1}\cdots \left(\frac{\gamma_{m}(W_n)}{m!}\right)^{k_{m-1}}H_{k-1}(x)
$$
where the summation is running over all tuples $\bar k=(k_1,...,k_{m-2})$
 of nonnegative integers so that $\sum_{j}j k_j\leq m-2$, and $k=k(k_1,...,k_{m-2})=3k_1+...+mk_{m-2}$. Note that $\gamma_j(W_n)=\Lambda_n^{(j)}(0)$ and so
  under Assumption \ref{GrowAssum} we have 
 $$
 \gamma_j(W_n)=O(\sig_n^{-(j-2)}),\, j=3,4,...,m+1.
 $$
Observe that $\Phi_{m,n}(x)$
can also be written in the form 
\begin{equation}\label{Phi def}
\Phi_{n,m}(x)=\Phi(x)-\varphi(x)\sum_{j=1}^{m-2}\sig_n^{-j}\textbf{H}_{j,n}(x)
\end{equation}
where 
\begin{equation}\label{Poly def 1}
\textbf{H}_{j,n}(x)=\sum_{\bar k\in A_{j}}\prod_{j=1}^{s(\bar k)-2}\left(\gamma_{j+2}(S_n)\sig_n^{-2}\right)^{k_j}H_{k-1}(x)
\end{equation}
and  $A_{j}$ is the set of all  tuples of nonnegative integers $(k_1,...,k_{s}), k_s\not=0$ for some $s=s(\bar k)\geq 1$ so that $\sum_{l}lk_l=j$ (note that when $j\leq m-2$ then $s\leq m-2$ since $k_s\geq1$). Moreover, $k=k(k_1,...,k_s)=3k_1+...+(s+2)k_s$.

\subsection{Expansions of functions with polynomially fast growing derivatives}

As explained in Section \ref{Intro2}, the following two results follow from Theorems \ref{BE1} and \ref{Edge1}.
\begin{corollary}\label{Cor}
Let $h:\bbR\to\bbR$ be  an a.e. differentiable function so that $H_m=\int\frac{|h'(x)|}{(1+|x|)^m}dx<\infty$.  
\vskip0.1cm
 Under Assumption \ref{GrowAssum} and \ref{DerAss}, for all  $1\leq r\leq m-2$ we have
$$
\left|\bbE[h(W_n)]-\int h(x)d\Psi_{r,n}(x)\right|\leq C_nH _mB_n^{-r},\,\, C_n\to 0.
$$
In particular, for every $q<m$ we have 
$$
\left|\bbE[W_n^q]-\int x^qd\Psi_{r,n}(x)\right|\leq R_m C_nB_n^{-r},\,  C_n\to 0.
$$
and 
$$
\left|\bbE[|W_n|^q]-\int |x|^q d\Psi_{r,n}(x)\right|\leq R_m C_nB_n^{-r},\,  C_n\to 0
$$ 
where $R_m$ depends only on $m$.
\end{corollary}

 \begin{remark}\label{h rem}
Corollary \ref{Cor}  yields appropriate expansions for the moments like in \cite{FL}, but in a more general setup, and for non-integer moments. Moreover, we are able to expand expectations of more general functions. 
 
In comparison with Corollary \ref{Cor}, in \cite{Bar} and \cite{RinRot} similar estimates were obtained for different classes of functions $h$. However, in general these results do not yield optimal rates $o(\sig_n^{-r})$ and they also hold only for $m$ times differentiable functions $h$ satisfying some regularity conditions. On the other hand, the results in  \cite{Bar} and \cite{RinRot} apply to sufficiently fast mixing sequences, which is a different class of processes than the ones considered in this paper\footnote{The Markov chains considered in Section \ref{MC IH} will satisfy both our assumptions and the mixing assumptions in \cite{Bar} and \cite{RinRot}, but our results hold for a different class of functions $h$, which might only be $C^1$ with a sufficiently fast decaying derivative. Moreover, we obtain optimal rates.}.
 \end{remark}
 
\subsection{Better than optimal CLT rates in transport distances}
The explicit formula for $\textbf{H}_{j,n}$ together with  \cite[Proposition 5.1]{Bobkov2018} yield the following result.

\begin{theorem}
Suppose $A_n=0$ and $B_n=\sig_n$.
Let  the conditions of Theorem \ref{Edge1} hold. 
Then for all $r=1,2,...,m-2$ and $p<m-1$ such that $r\geq p$ we have the following. 
 If $\gamma_j(S_n)/\sig_n^j=o(\sig_n^{-\frac{r}{p}})$  for all $3\leq j\leq 2+r/p$
 then 
$$
W_p(dF_n,d\Phi)=o(\sig_n^{-\frac{r}{p}}).
$$

In particular (by taking $r=p$), 
$$
W_p(dF_n,d\Phi)=o(\sig_n^{-1})
$$
if $\gamma_3(S_n)=\bbE[S_n^3]=o(\sig_n^{2})$.
\end{theorem}
Recall that in \cite{Esseen1956}, for partial sums with iid summands (which do not takes values on a lattice) it was shown that we have uniform CLT rates of order $o(\sig_n^{-1})$ if the third moment of $X_1-\bbE[X_1]$ vanishes. In \cite{DolgHaf} we provided an appropriate version of this phenomena for additive functionals of uniformly elliptic inhomogeneous Markov chains (where the condition about the third moment is replaced by $\bbE[S_n^3]=o(\sig_n^2)$). 
The above theorem provides similar result but in the metric $W_p$.

\subsection{Applications to Gaussian coupling: A Berry-Esseen theorem in $L^p$}\label{App}
The following result is a consequence of Theorem \ref{BE1} together with \cite[Lemma 11.1]{Bobkov2018}. 
\begin{corollary}\label{ASIP}
Let  Assumption \ref{GrowAssum} hold with some $m$ and
suppose that $B_n$ is monotone increasing. Let $(Z_j)$ be a sequence of 
 centered normal random variables 
so that $\text{Var}(Z_j)=B_{j}^2-B_{j-1}^2$, where $B_0=0$.

Set $X_i=S_i-S_{i-1}, S_0=0$. Then for every $n$ there  is a coupling of $(X_1,...,X_n)$ with $(Z_1,...,Z_n)$ so that 
$$
\left\|S_n-\sum_{j=1}^{n}Z_j\right\|_{m-1}\leq C
$$
where $C$ is a constant which does not depend on $n$.

\end{corollary}
\begin{remark}
In the stationary case  we can just take $B_n=\sig n$ for some $n$ and $\sig\geq 0$. In Sections \ref{RDS} and \ref{SDS sec} we will show that for certain classes of random and sequential dynamical systems we have $\sig_n^2=a_n+O(1)$ for an increasing sequence $(a_n)$. Thus, we can always take $B_n=\sqrt{a_n}$ in the above theorem.
We will also see  in Section \ref{MC IH} that $\sig_n^2=a_n+O(1)$  for an increasing sequence $(a_n)$
 for uniformly  bounded additive functionals of uniformly elliptic inhomogeneous Markov chains.
 \end{remark}

\begin{proof}[Proof of Corollary \ref{ASIP}]
For a fixed $n$, let $Y_j=Y_{j,n}=Z_j/B_n; j\leq n$. Then $Y_j$ is a centered normal random variable with variance $s_{j,n}=\frac{B_{j}^2-B_{j-1}^2}{B_n^2}$. Note that  $Y=\sum_{j=1}^{n}Y_{j}$ has the standard normal law.
By Applying \cite[Lemma 11.1]{Bobkov2018} with $X_i=S_{i}-S_{i-1}$ (where $S_0=0$) and the above $Y_i$, together with Theorem \ref{BE1}, 
we see that we can couple $(X_1,...,X_n)$ with $(Z_1,...,Z_n)$ so that 
$$
\left\|S_n-\sum_{j=1}^{n}Z_j\right\|_{m-1}\leq C.
$$
\end{proof}


\section{Reduction to the case $A_n=0$ and $B_n=\sig_n$: proof of Theorem \ref{Edge2} relying on the self-normalized case}\label{Reduc}
In this section we will prove Theorems \ref{Edge2} based on the validity of these theorems in the self-normalized case when $B_n=\sig_n$ and $A_n=0$. Let $W_n=\frac{S_n}{\sig_n}$ and $Z_n=\frac{W_n-A_n}{\sig_n}$. Then for every $x\in\bbR$,
$$
\bbP(Z_n\leq x)=\bbP(W_n\leq a_nx+v_n)
$$
where $a_n=\frac{B_n}{\sig_n}=1+O(\sig_n^{-1})$ and $v_n=\frac{A_n}{\sig_n}=O(\sig_n^{-1})$. Therefore, 
all the results stated in Theorem \ref{Edge2} hold true  with the generalized distribution function 
$$
\tilde \Phi_{m,n}(x)=\Phi_{m,n}(a_n x+v_n)
$$
instead of  $\Phi_{m,n}$,
where $\Phi_{m,n}$ is defined in \eqref{Phi def}.
The next step of the proof will be to pass from $\tilde \Phi_{m,n}(x)$ to a function defined similarly to   $\Phi_{m,n}(x)$ but possibly with different polynomials.
This is achieved in the following result.

\begin{lemma}
There are polynomials $U_{r,n}$ with bounded coefficients and degrees depending only on $r$ so that 
\begin{equation}\label{EST1}
\left|\tilde \Phi_{m,n}(x)-\left(\Phi(x)+\varphi(x)\sum_{r=1}^{m-2}B_n^{-r}U_{r,n}(x)\right)\right|\leq C_m(1+|x|^{e_m})e^{-x^2/4}\sig_n^{-(m-1)}
\end{equation}
where $C_m$ and $e_m$ are constants which depend only on $m$. 
\end{lemma} 
\begin{proof}
First, for any choice of polynomials $U_{r,n}$, if we choose $|x|\geq \sig_n^{\ve}$ for some $\ve\in(0,1)$ then, regardless of the choice of $U_{r,n}$, since $a_n\to1$ and $v_n\to 0$ both terms inside the absolute value on the left hand side of \eqref{EST1} are at most of of order $|x|^{e_m}e^{-x^2/3}$ for some $e_m$.
Let us write 
$$
|x|^{e_m}e^{-x^2/3}=|x|^{e_m}e^{-x^2/12}e^{-x^2/4}\leq |x|^{e_m}e^{-\sig_n^{2\ve}/12}e^{-x^2/4}\leq C|x|^{e_m}e^{-x^2/4}\sig_n^{-(m-1)}.
$$
Thus, by it is enough to find polynomials which satisfy \eqref{EST1} for all $n$ and $x$ so that $|x|\leq \sig_n^{\ve}$, where $\ve$ can be chosen to arbitrarily close to $0$.

Let us fix some $n$ and $x$ so that  $|x|\leq \sig_n^{\ve}$. Let us first suppose that $|x|\leq C$ for some constant $C>0$. Let us write 
$$a_nx+v_n=x+(u_nx+v_n):=x+\eta$$
where $u_n=a_n-1=\frac{B_n}{\sig_n}-1$ and $\eta=\eta(n,x)=u_nx+v_n$. Then our assumption about the sequences $B_n$ and $\sig_n$ insures that $u_n=\frac{D_n}{B_n}$ for some bounded sequence $D_n$. 
Now, using that $\varphi^{(k)}(x)=(-1)^k\varphi(x)H_k(x)$, by the Lagrange form of the Taylor remainders of $\Phi$ around the point $x$  for all $s$ we have
$$
\Phi(a_nx+v_n)=\Phi(x+\eta)=\Phi(x)+\varphi(x)\sum_{k=1}^{s}\frac{(-1)^{k-1}H_{k-1}(x)}{k!}\eta^k+O(\eta^{s+1}).
$$
Now, since $|x|\leq C$ our assumption about $A_n$ and $B_n$
guarantee that $\eta=O(\sig_n^{-1})$ and so  the above remainder is of order $e^{-x^2/4}(1+|x|^{m})^{-1}\sig_n^{-(m-1)}$ if $s+1>m-1$. Similarly,  
$$
\varphi(a_nx+v_n)=\varphi(x+\eta)=\varphi(x)+\varphi(x)\sum_{k=1}^{s}\frac{(-1)^{k}H_{k}(x)}{k!}\eta^k+O(\eta^{s+1})
$$
and the remainder is of order  $e^{-x^2/4}(1+|x|^{m})^{-1}\sig_n^{-(m-1)}$. Finally, let us expand the polynomials. Now, let us write $\textbf{H}_{r,n}(x)=\sum_{k=0}^{d_r}a_{k,r,n}x^k$, where $\textbf{H}_{r,n}$ are defined in \eqref{Poly def 1}. Then 
\begin{equation}\label{H pass}
\textbf{H}_{r,n}(a_nx+v_n)=\textbf{H}_{r,n}(x+\eta)=\sum_{q=0}^{d_r}\left(\sum_{k=q}^{d_r}\binom{k}{q}\eta^{k-u}\right)x^{q}.
\end{equation}
Since $u_n=\frac{D_n}{B_n}$, $v_n=O(\sig_n^{-1})$ and $|x|\leq C$ we see that
$\eta^j$ has the form $\eta^j=B_n^{-j}(D_n x+L_n)^j$ for some bounded $L_n$ and $D_n$. Thus $\eta^j$ is a polynomial in $x$ whose coefficients are of order $B_n^{-j}$. Combing this with the previous estimates and an ignoring terms of order $B_n^{-u}$ for $u>(m-1)$ we see obtain  \eqref{EST1} with some polynomials $U_{r,n}$ on bounded domains of $x$.

Next, let us show that the above three Taylor estimates yield \eqref{EST1} with the same polynomials on  domains of the form $\{C\leq |x|\leq \sig_n^{\ve}\}$, with $C$ large enough. To achieve that we use the Lagrange form of the Taylor remainder to get that 
$$
\Phi(a_nx+v_n)=\Phi(x+\eta)=\Phi(x)+\varphi(x)\sum_{k=1}^{s}\frac{(-1)^{k-1}H_{k-1}(x)}{k!}\eta^k+\frac{(-1)^s\eta^{s+1}}{(s+1)!}H_{s}(x+\zeta_1)\varphi(x+\zeta_1)
$$
and 
$$
\varphi(a_nx+v_n)=\varphi(x+\eta)=\varphi(x)+\varphi(x)\sum_{k=1}^{s}\frac{(-1)^{k}H_{k}(x)}{k!}\eta^k+\frac{(-1)^s\eta^{s+1}}{(s+1)!}H_{s+1}(x+\zeta_2)\varphi(x+\zeta_2)
$$
where $\max(|\zeta_1|, |\zeta_2|)\leq\eta$. Next, since $|x|\leq \sig_n^\ve$ we get that $|\eta|=O(\sig_n^{-(1-\ve)})$ and hence when $(s+1)(1-\ve)>m-1$ we have $|\eta|^{s+1}=o(B_n^{-(m-1)})$. Next, since 
$$
|\zeta_i|\leq|\eta|=o(1)
$$
and $x^2\leq \sig_n^{2\ve}$, if $\ve$ is small enough we get that 
$$
\varphi(x+\zeta_i)\leq e^{-x^2/3}
$$
and so for $i=1,2$
$$
|H_{s+i}(x)|\varphi(x+\zeta_i)\leq (1+|x|^{u_s})e^{-x^2/4}
$$
for some constant $u_s$.
By using again \eqref{H pass} and expressing $\eta^j$ as a polynomial in $x$ whose coefficients are of order $\sig_n^{-j}$,
the proof of \eqref{EST1} is completed now also in the case $C\leq |x|\leq \sig_n^\ve$.
\end{proof}

\section{Proofs of the main results in the self normalized case}\label{SN}
\subsection{The Edgeworth polynomials}\label{Form}
Recall that $\gamma_j(W)$ denotes the $j$-th cumulant of a random variable $W$ with finite absolute $j$'s moment. Then for all $3\leq j\leq m+1$,
\begin{equation}\label{Cum}
\gamma_j(W_n)=\Lambda_n^{(j)}(0)=O(\sig_n^{-(j-2)})
\end{equation}
where $\Lambda_n$ comes from the main Assumption \ref{GrowAssum}.
Let us consider the following polynomials 
\begin{equation}\label{P def}
P_{m,n}(z)=\sum_{\bar k}\frac1{k_1!\cdots k_{m-2}!}\left(\frac{\gamma_3(W_n)}{3!}\right)^{k_1}\cdots \left(\frac{\gamma_{m}(W_n)}{m!}\right)^{k_{m-2}}z^{3k_1+...+mk_{m-2}}
\end{equation}
where the summation runs over the collection of $m-2$ tuples of nonnegative integers $(k_1,...,k_{m-2})$ that are not all $0$ so that $\sum_{j} jk_j\leq m-2$. Let $\nu_{m,n}$ be the signed measure on $\bbR$ whose Fourier transform is 
$$
g_{m,n}(t)=e^{-t^2/2}(1+P_{m,n}(it)).
$$
Let $H_k(z)$ be the $k$-th Hermite polynomial, which is defined through the identity $(-1)^k H_k(x)\varphi(x)=\varphi^{(k)}(x)$.
Then $\mu_{m,n}$ is absolutely continuous with respect to the Lebesgue measure with density
$$
\varphi_{m,n}(x)=\varphi(x)\left(1+\sum_{\bar k}\frac1{k_1!\cdots k_{m-2}!}\left(\frac{\gamma_3(W_n)}{3!}\right)^{k_1}\cdots \left(\frac{\gamma_{m}(W_n)}{m!}\right)^{k_{m-2}}H_k(x)\right)
$$
where $\varphi(x)=\frac{1}{\sqrt{2\pi}}e^{-x^2/2}$ is the standard normal density function and $k=k(k_1,...,k_{m-2})=3k_1+...+mk_{m-2}$. Using that $H_k(x)\varphi(x)=-\left(H_{k-1}(x)\varphi(x)\right)'$ we see that the corresponding generalized distribution function is given by
\begin{equation}\label{Phi m n def 1}
\Phi_{m,n}(x)=
\end{equation}
$$\int_{-\infty}^{x}\varphi_{m,n}(x)dx=\Phi(x)-\varphi(x)\sum_{\bar k}\frac1{k_1!\cdots k_{m-2}!}\left(\frac{\gamma_3(W_n)}{3!}\right)^{k_1}\cdots \left(\frac{\gamma_{m}(W_n)}{m!}\right)^{k_{m-2}}H_{k-1}(x).
$$ 
Note that $\Phi_{m,n}(-\infty)=0$ and 
$$
\Phi_{m,n}(\infty)=\int_{-\infty}^{\infty}\varphi_{n,m}(x)dx=g_{m,n}(0)=1.
$$
Notice also that for $m=2$ we have $P_{2,n}=0$ and so $\nu_{2,n}$ is the standard normal law and $\Phi_{2,n}=\Phi$ is the standard normal distribution function.

Next, observe that the function $\Phi_{n,m}(x)$ can also be written in the form
\begin{equation}\label{BoldP}
\Phi_{n,m}(x)=\Phi(x)-\varphi(x)\sum_{r=1}^{m-2}\sig_n^{-r}\textbf{H}_{r,n}(x)
\end{equation}
where
\begin{equation}\label{Poly def 1.1}
\textbf{H}_{r,n}(x)=\sum_{\bar k\in A_{r}}C_{\bar k}\prod_{j=1}^{s}\left(\gamma_{j+2}(S_n)\sig_n^{-2}\right)^{k_j}H_{k-1}(x)
\end{equation}
and $A_{r}$ is the set of all  tuples of nonnegative integers $\bar k=(k_1,...,k_{s}), k_s\not=0$ for some $s=s(\bar k)\geq 1$ so that $\sum_{j}jk_j=r$ (note that when $r\leq m-2$ then $s\leq m-2$ since $k_s\geq1$). Moreover, $k=3k_1+...(s+2)k_s$ and 
 $$
 C_{\bar k}=\prod_{j=1}^{s}\frac{1}{k_j!(j+2)^{k_j}}.
 $$
 We note that  the polynomials
$\textbf{H}_{r,n}$ have bounded coefficients (because of \eqref{Cum}) and that  their degrees does not depend on $n$. The  polynomails $\textbf{H}_{r,n}$ coincide with the ``Edgeworth polynomials" defined in \cite{DolgHaf} (denoted there by $P_{r,n}$).

We also note that the  Fourier transform of the derivative of $\Phi_{m,n}(x)$ has the form 
$$
g_{m,n}(x)=e^{-x^2/2}\left(1+\sum_{r=1}^{m-2}\sig_n^{-r}\textbf{P}_{j,n}(x)\right)
$$
where 
$$
\textbf{P}_{r,n}(x)=\sum_{\bar k\in A_{r}}C_{\bar k}\prod_{j=1}^{s}\left(\gamma_{j+2}(S_n)\sig_n^{-2}\right)^{k_j}(ix)^{3k_1+...+(s+2)k_{s}}.
$$


\subsection{Proof of Theorem  \ref{Edge2}}
The proof of Theorem \ref{Edge2} will make use of the following two results from \cite{NonU BE}.


The first one is \cite[Proposition 6]{NonU BE}:
\begin{proposition}\label{Prp17.1}
Under Assumption \ref{GrowAssum}, 
if $|t|\geq C\sig_n^{1/3}$ for some $C>0$ then for  $p=0,1,2,...,m$ we have
$$
\left|g_{m,n}^{(p)}(t)\right|\leq C_m\sig_n^{-(m+1)}e^{-c_0t^2}
$$
where $C_m$ depends only on $C,m$ and the constants in Assumption \ref{GrowAssum} and $c_0\in(0,1/2)$ is a constant (which, upon increasing $C_m$, can be made arbitrarily close to $1/2$).
\end{proposition}

The second one is \cite[Proposition 7]{NonU BE}
\begin{proposition}\label{Prp3}
Under Assumption \ref{GrowAssum}, 
there are constants $c,c_0>0$ and $C_m>0$ so that if $|t|\leq c\sig_n$ then for $p=0,1,2,...,m$ we have 
$$
\left|\frac{d^p}{dt^p}\left(f_n(t)-g_{m,n}(t)\right)\right|\leq C_m e^{-c_0t^2}\min\{1, |t|^{m+1-p}\}\sig_n^{-(m-1)}.
$$
\end{proposition}


\subsection{Proof of Theorem \ref{Edge2} relying on relying on Propositions \ref{Prp17.1} and \ref{Prp3}}

\begin{lemma}\label{NonULemma}
Under Assumption \ref{GrowAssum},
for all $x\in\bbR$ and $3\leq p\leq m$ we have
$$
\left|\Phi_{p,n}(x)-\Phi_{p-1,n}(x)\right|\leq C_m(1+|x|)^{-m-1}\sig_n^{-(p-2)}.
$$
As a consequence, for all $2\leq m_1<m$,
$$
\left|\Phi_{m,n}(x)-\Phi_{m_1,n}(x)\right|\leq A_m(1+|x|)^{-m-1}\sig_n^{-(m_1-1)}.
$$
Here $A_m,C_m$ are constants that do not depend on $n$ or $x$.

\end{lemma}

\begin{proof}
Notice that the difference between the derivatives of $\Phi_{p,n}(y)$ and $\Phi_{p-1,n}(y)$ has the form
$$
\varphi_{p,n}(y)-\varphi_{p-1,n}(y)=\frac{e^{-y^2/2}}{\sqrt{2\pi}}\left(\sum_{\bar k}\frac{1}{\prod_{j=1}^{p-2}k_j!(j!)^{k_j}}\prod_{j=1}^{p-2}\left(\gamma_{j+2}(W_n)\right)^{k_j}H_k(y)\right)
$$ 
where the summation is running over all tuples $\bar k=(k_1,...,k_{p-2})$
 of nonnegative integers so that $\sum_{j}j k_j=p-2$. 
 Using \eqref{GAs} with $t=0$, we see that the product of the cumulants in the above difference is of order $\sig_n^{-(p-2)}$. Therefore,
$$
\left|\Phi_{p,n}(x)-\Phi_{p-1,n}(x)\right|\leq \int_{-\infty}^{x}|\varphi_{p,n}(y)-\varphi_{p-1,n}(y)|dy\leq 
C\sig_n^{-(p-2)}\int_{-\infty}^{x}e^{-y^2/2}(1+|y|^s)dy 
$$
where $s=s_m$ is some positive integer. This proves that the lemma when $x\leq 0$ since the integral on the right hand side decays exponentially fast in $|x|$ as $x\to-\infty$.
To prove the statement for positive $x$'s we recall that $\Phi_{j,n}(\infty)=1$,
and so 
$$
\Phi_{p,n}(x)-\Phi_{p-1,n}(x)=(1-\Phi_{p-1,n}(x))-(1-\Phi_{p,n}(x))$$
$$
=(\Phi_{p-1,n}(\infty)-\Phi_{p-1,n}(x))-(\Phi_{p,n}(\infty)-\Phi_{p,n}(x))=
\int_{x}^{\infty}\left(\varphi_{p-1,n}(y)-\varphi_{p,n}(y)\right)dy.
$$
Using this inequality  the proof proceeds similarly to the case when $x\to-\infty$.
\end{proof}

\begin{proof}[Proof of Theorem \ref{Edge2}]
By the results described in \cite[Ch. VI, Lemma 8]{Petrov72}, if $F$ is a distribution function and $G$ is a  generalized distribution function so that $G(-\infty)=F(-\infty)=0$,  $G(\infty)=F(\infty)=1$, $|G'(x)|\leq K(1+|x|)^{-s}$ for some constant $K$ and all $x\in\bbR$ and
$$
\int_{-\infty}^{\infty}|x|^s|d(F(x)-G(x))|<\infty
$$
then for all $x\in\bbR$ and $T>1$,
\begin{equation}\label{Ess}
|F(x)-G(x)|\leq c(s)(1+|x|)^{-s}\left(\int_{-T}^{T}\left|\frac{f(t)-g(t)}{t}\right|dt+\int_{-T}^{T}\left|\frac{a_s(t)}{t}\right|dt+\frac{K}{T}\right)
\end{equation}
where $f(t)=\int_{-\infty}^{\infty} e^{itx}dF(x)$, $g(t)=\int_{-\infty}^{\infty} e^{itx}dG(x)$ and 
$$
a_s(t)=\int_{-\infty}^{\infty}e^{itx}d\left(x^s(F(x)-G(x))\right).
$$
Here $c(s)$ is an absolute constant which depends only on $s$.
Notice that for $t\not=0$ we have (see \cite{Bobkov2016}),
\begin{equation}\label{aaok}
a_s(t)=(i)^{-s}\frac{d^s}{dt^s}\,\frac{a_0(t)}{t}=i^{-s}\int_{0}^1\left(a_0^{(s)}(t)-a_0^{(s)}(\eta t)\right)s\eta^{s-1}d\eta.
\end{equation}
Note also that $a_0(y)=f_n(y)-g_{m,n}(y)$. Moreover, observe that for every measurable nonnegative function $q:\bbR\to\bbR$ and all $T>0$ and $s\geq1$ we have 
$$
\int_{0}^{T}\frac 1t\int_{0}^1 q(\eta t)s\eta^{s-1}d\eta dt=\int_{0}^{T}q(x)\left(\frac 1x-\frac 1T\cdot (x/T)^{s-1}\right)dx\leq\int_{0}^{T}\frac{q(x)}{x}dx.
$$ 
Indeed, the above follows by making the change of variables $t\eta=x$ in the inner integral and then changing the order of  integration.
Therefore, using also \eqref{aaok} we see that
\begin{equation}\label{aaaok}
\int_{-T}^{T}\left|\frac{a_s(t)}{t}\right|dt\leq 2\int_{-T}^{T}\left|\frac{a_0^{(s)}(t)}{t}\right|dt.
\end{equation}

To prove the non-uniform Edgeworth expansions (Theorem \ref{Edge2}) of order $r\leq m-2$ and power $m$, by using Lemma \ref{NonULemma} it is enough to consider the case  $r=m-2$. Let us take again $F$ to be the distribution function of $W_n$ and $G=\Phi_{m,n}$. Let us also take  $s=m$.
Let us fix some $\ve\in(0,1)$ and let 
 $T=B\sig_n^{m-2}$
for some $B>\frac 1\ve$.  We thus see that the contribution of the term $c(s)K/T$ is at most $c\ve \sig_n^{-(m-2)}=c\ve\sig_n^{-r}$  for some constant $c$ which does not depend on $\ve$. Hence, it remains to show that for this fixed $B$ all the rest of the expressions on the right hand side of \eqref{Ess} are $o(\sig_n^{-(m-2)})=o(\sig_n^{-r})$.
To estimate the first integral in the brackets on the right hand side of \eqref{Ess}, we split  it to the integrals on two domains. The first is a domain of the form $\{|t|=O(\sig_n)\}$, and the corresponding integral is of order $O(\sig_n^{-(m-1)})$ because of Proposition \ref{Prp3}. The second domain is of the form $\{c_0\sig_n\leq |t|\leq T\}$, and to estimate the corresponding integral we use Proposition \ref{Prp17.1} together with Assumption \ref{DerAss}, which yield that 
$$
\int_{-T}^{T}\left|\frac{f_n(t)-g_{m,n}(t)}{t}\right|dt=o(\sig_n^{-(m-2)})=o(\sig_n^{-r}).
$$
To estimate the second integral in the brackets of the right hand side of \eqref{Ess}, we first use \eqref{aaaok}, and then we split the integral
$$
\int_{-T}^{T}\left|\frac{a_s(t)}{t}\right|dt
$$
into two domains. The first is of the form $\{|t|\leq c_0\sig_n\}$ for a sufficiently small $c_0$. By taking $p=m$  in Proposition \ref{Prp3} we see that the contribution of this domain of integration is $O(\sig_n^{-(m-1)})$. The second domain has the form $\{c_0\sig_n\leq |t|\leq T\}$, 
and by combining Proposition \ref{Prp17.1} with  Assumption \ref{DerAss}, we see that the contribution coming from this domain is $o(\sig_n^{-(m-2)})$. We conclude that
$$
\int_{-T}^{T}\left|\frac{a_s(t)}{t}\right|dt=o(\sig_n^{-(m-2)})=o(\sig_n^{-r})
$$ 
and the proof of Theorem \ref{Edge2} is complete.
\end{proof}

\subsection{Proof of Theorem \ref{BE1} via smoothing}
The  proof begins with a classical smoothing argument.
Similarly to \cite{Bobkov2018}, let $\xi_m$ be a centered random variable with finite $m$ absolute moments  whose characteristic function $h_m$ is supported on some interval $[-a,a]$ for some $a>0$. For example, we can first take $\tilde \xi_m$ to be the random variable with the density $w_m(x)=\Lambda^{-1}\left(\frac{\sin x}{x}\right)^{2m}$, where $\Lambda$ is an appropriate normalizing constant (so that the characteristic function is proportional to $(1-|t|)_+$). Then we can normalize $\tilde \xi_m$ so that its second moment is $1$. Let us denote the normalized version by $\xi_m$. By enlarging the underlying probability space if necessary, we can always assume that $\xi_m$ is defined on this space and it is independent of $S_n$.
Fix some constant $c>0$, let $n$ be so that $c\sig_n>1$ and set 
$$
\tilde S_n=\sqrt{1-(c\sig_n)^{-2}}S_n+c^{-1}\xi_m.
$$
Then, because $\xi_m$ is independent of $S_n$ we have 
\begin{equation}\label{Prod}
\tilde\psi_n(t):=\bbE[e^{it\tilde S_n}]=\psi_n(t)h(t/c)
\end{equation}
and so
the characteristic function $\tilde \psi_n$ of $\tilde S_n$ vanishes outside the interval $[-ca, ca]$. Notice also that 
\begin{equation}\label{vvvv}
\sig_n^2=\|S_n\|_{L^2}^2=\|\tilde S_n\|_{L^2}^2.
\end{equation}
In view of \eqref{Prod} and \eqref{vvvv}, it is also clear that Assumption \ref{GrowAssum} is satisfied for $\tilde S_n$ if it holds for $S_n$, and with the same $m$.  
Now, if $c$ is small enough then Assumption \ref{DerAss} trivially holds true for $\tilde S_n$ for every $m$, and so, by applying Theorem \ref{Edge2} with $\tilde S_n$ instead of $S_n$ (in the self normalized case) we see that, with $\tilde F_n(x)=\bbP(\tilde S_n/\sig_n\leq x)$ we have
$$
\left|\tilde F_n(x)-\tilde\Phi_{m-2,n}(x)\right|\leq C(1+|x|)^{-m}\sig_n^{-(m-1)}
$$
where $\tilde \Phi_{m-2,n}$ is defined similarly to $\Phi_{m-2,n}$ but with the cumulants of $\tilde S_n$ instead of the cumulants of $S_n$, and $C$ is some constant.
By applying \cite[Corollary 3.2]{Bobkov2018} we see that for every $p<m-1$ we have
$$
W_{p}(d\tilde F_n, d\tilde \Phi_{m-2,n})=O(\sig_n^{-1}).
$$
Next, by applying  \cite[Proposition 5.1]{Bobkov2018} (with $\ve=O(\sig_n^{-1})$) together with Lemma \ref{NonULemma}  we see that 
$$
W_p(d\tilde \Phi_{m-2,n}, d\Phi)=O(\sig_n^{-1}).
$$
Finally, by the definition of $W_p$ we have 
$$
W_p(dF_n, d\tilde F_n)\leq \|\tilde S_n-S_n\|_{L^p}\leq (1-\sqrt{1-(c\sig_n)^{-2}})\|S_n\|_{L^p}+\|\xi_m\|_{L^p}(c\sig_n)^{-1}\leq 
$$
$$
(c\sig_n)^{-2}\|S_n\|_{L^p}+\|\xi_m\|_{L^p}(c\sig_n)^{-1}.
$$
In order to to complete the proof of Theorem \ref{BE1}, we apply the last three $W_p$-estimates together with the triangle inequality for the metric $W_p$ and the following result:

\begin{lemma}\label{MomLemm}
Under Assumption \ref{GrowAssum} we have that for all $p\leq m$,
$$
\|S_n\|_{L^p}\leq C_p\sig_n
$$
for some constant $C_p$.
\end{lemma}
\begin{proof}[Proof of Lemma \ref{MomLemm}]
Let $k\in\{m,m+1\}$ be an even number. Then 
$$
\|S_n\|_{L^p}\leq \|S_n\|_{L^k}=\left(\bbE[S_n^k]\right)^{1/k}.
$$ 
Next, by \cite[(1.53)]{SaulStat} (which in our case is also a  consequence of Fa\`a Di Bruno's formula) we see that $\bbE[S_n^k]$ is a linear combination (with coefficients not depending on $n$) of products of the form $\prod_{j=1}^{d}(\gamma_{k_j}(S_n))^{r_j}$ with $\sum_{j=1}^d r_j k_j=k$ and $r_j>0$ ($k_j$ and $r_j$ are nonnegative integers). Note that $\Gamma_1(S_n)=\bbE[S_n]=0$, and so we can assume that $k_j>1$ for all $j$. Now, by Assumption \ref{GrowAssum} we have that $\gamma_{k_j}(S_n)=O(\sig_n^2)$ for all $j$ (as $k_j\leq k\leq m+1$). Thus the contribution coming from the above product is $O(\sig_n^{2\sum_j r_j})$. Notice now that since $k_j\geq 2$ we have $2\sum_j r_j\leq \sum_j  r_j k_j=k$. We thus conclude that
$$
\bbE[S_n^k]\leq C_k\sig_n^k
$$
as claimed.
\end{proof}

\section{Forms of stationarity-stationary expansions}\label{WeakForm}

In what follows we will verify the following weak form of stationary
\begin{assumption}[A weak form of stationarity]\label{AssStat}
There are numbers $p_k,q_k, k\geq 2$ and $\delta\in(0,1)$ such that for all $j\geq 2$ we have 
 \begin{equation}\label{Relat1.0}
  \gamma_j(S_n)=np_j+q_j+O(\del^n).
 \end{equation}
\end{assumption}

\subsection{Stationary expansions in the centralized self-normalized case}
The following result shows that under Assumption \ref{AssStat} we get stationary expansions.
\begin{proposition}\label{pass1}
Under Assumption \ref{AssStat}
there are polynomials $\textbf{H}_j$ whose coefficients are polynomial functions of  $\al_j=q_{j+2}-\frac{q_2p_{j+2}}{p_2}$ and $\beta_j=\frac{p_{j+2}}{p_2}$ for $j\leq m-2$ so that
\begin{equation}\label{P1}
\Phi_{n,m}(x)=\Phi(x)-\varphi(x)\sum_{j=1}^{m-2}\sig_n^{-j}\textbf{H}_{j}(x)+e^{-x^2/2}R_n(x)\sig_n^{-(m-1)}
\end{equation}
where $R_n$ is a polynomials with bounded coefficients and degree depending only on $m$.
As a consequence, if Theorems  \ref{Edge2}, \ref{BE1} and \ref{Edge1} hold true with $A_n=0$, $B_n=\sig_n$  and $\Psi_{k,n}=\Phi_{k+2,n}$ given by \eqref{deff}, then the same results 
 hold true for every $m$ with the function $\mathbf{\Phi}_{m,n}$ given by 
\begin{equation}\label{bold Phi}
\mathbf{\Phi}_{m,n}(x)=\Phi(x)-\varphi(x)\sum_{j=1}^{m-2}\sig_n^{-j}\textbf{H}_j(x)
\end{equation}
instead of the function $\Phi_{m,n}$.
\end{proposition}
We refer the readers to Remark \ref{RemH} for an explicit formula for $\textbf{H}_r$.

\begin{proof}[Proof of Proposition \ref{pass1}]
Let us recall that
$$
\Phi_{m,n}(x)=\Phi(x)-\varphi(x)\sum_{j=1}^{m-2}\sig_n^{-j}\textbf{H}_{j,n}(x)
$$
where  $\textbf{H}_{r,n}(x)$ are defined in \eqref{Poly def 1}.
Now, by Assumption \ref{AssStat}, we have
$$
\gamma_{j+2}(S_n)=np_{j+2}+q_{j+2}+O(\del^n)
$$
and
$$n/\sig_n^2=\frac{1}{p_2}-\frac{q_2}{p_2 \sig_n^2}+O(\del^n).
$$
Thus,
$$
\gamma_{j+2}(S_n)\sig_n^{-2}=\sig_n^{-2}\left(q_{j+2}-\frac{q_2p_{j+2}}{p_2}\right)+\frac{p_{j+2}}{p_2}+O(\del^n):=\al_j\sig_n^{-2}+\beta_j+O(\del^n).
$$
Since $\sig_n$ is of order $\sqrt n$ we can just disregard the $O(\del^n)$ term and consider the polynomials 
$$
\tilde H_{r,n}(x)=\sum_{\bar k\in A_{r}}\prod_{j=1}^{s(\bar k)}\left(\al_j\sig_n^{-2}+\beta_j\right)^{k_j}H_{k-1}(x)
$$
instead of $\textbf{H}_{r,n}$. By expanding the brackets 
$$
\left(\al_j\sig_n^{-2}+\beta_j\right)^{k_j},
$$
rearranging the negative powers of $\sig_n$ and omitting all the terms which involve higher than $m-2$ powers of $\sig_n^{-1}$ we obtain \eqref{P1}.
\end{proof}

\begin{remark}\label{RemH}
Proceeding as  at the end of the proof, we get the following formula for $\textbf{H}_r$:
\begin{equation}\label{Poly 2 def}
\textbf{H}_r(x)=\sum_{\bar k\in\bigcup_{j=1}^{r}A_j}\left(\sum_{u\leq \frac12(m-r-2)}\,\sum_{\bar\ell\in A_{u,\bar k}}\prod_{j=1}^{s(\bar k)}\left(\binom{k_j}{\ell_j}\al_j^{\ell_j}\beta_j^{k_j-\ell_j}\right)\right)H_{k-1}(x)
\end{equation}
where given a tuple $\bar k=(k_1,...,k_{s(\bar k)})$ we have that $A_{u,\bar k}$ is the set of all  tuples $(\ell_1,...,\ell_{s(\bar k)})$  of nonnegative integers so that $\sum_j\ell_j=u$ and $\ell_j\leq k_j$, and as before $k=k(\bar k)=\sum_j(j+2)k_j$.
\end{remark}

\subsection{Stationary expansions with the $\sqrt n$ normalization}
The stationary expansions (i.e. when the polynomials do not depend on $n$) in the case when $B_n=\sqrt{p_2 n}$ will follow from the following result.
 
\begin{proposition}\label{pass2}
Under Assumption \ref{AssStat}, 
there are polynomials $\bar H_j$   so that all the results stated in the previous sections (under the appropriate assumptions) hold true for every $m$ with $B_n=\sqrt{p_2 n}$ and with the function $\bar\Phi_{m,n}$ given by 
$$
\bar\Phi_{m,n}=\Phi(x)-\varphi(x)\sum_{j=1}^{m-2}n^{-j/2}\bar H_j(x)
$$
instead of the function $\Phi_{m,n}$ (and also with the latter function $\mathbf{\Phi}_{m,n}(x)$).
\end{proposition}

\subsection{Proof of Proposition \ref{pass2}}

We first need the following result.
\begin{lemma}\label{L}
There are polynomials $\tilde H_j$ whose coefficients are polynomial functions of $q_k$ and $p_k$ for $k\leq 3j$ so that
\begin{equation}\label{P2}
\Phi_{m,n}(x)=\Phi(x)-\varphi(x)\sum_{j=1}^{m-2}n^{-r/2}\tilde H_{j}(x)+e^{-x^2/2}R_n(x)n^{-(m-1)/2}
\end{equation}
where $R_n$ is a polynomial with bounded coefficients and degree depending only on $m$. 
\end{lemma}
\begin{proof}
Starting from \eqref{P1},
we can expand $\sig_n^{-r}$ in powers of $n^{-1/2}$ using 
that $\sig_n^2=np_2+q_2+O(\del^n)$. Then we can plug in the resulting expression (without the error term) inside \eqref{P1} and obtain the desired result after rearranging the powers of $\sig_n^{-1}$. 
\end{proof}

\subsection{Completing the proof of Proposition \ref{pass2}}
First, by the stationary non-uniform expansions in the centered self-normalized case and Lemma \ref{L} and arguing as in the proof of Proposition \ref{pass1}, all the result described in Section \ref{Main} hold true with the random variable $W_n=S_n/\sig_n$ and the function
$$
\check\Phi_{m,n}(x)=\Phi(x)-\varphi(x)\sum_{j=1}^{m-2}n^{-r/2}\tilde H_{j}(x)
$$ 
instead of $\Phi_{m,n}$. To complete the proof let first explain how to pass from $W_n$ to $\textbf{W}_n=\frac{S_n-A_n}{\sqrt{p_2 n}}$.
 First, let $\rho_n=\frac{\sig \sqrt n}{\sig_n}$, where $\sig=\sqrt{p_2}>0$, and
$$
\textbf{F}_n(x)=\bbP(\textbf{W}_n\leq x)=\bbP(W_n\leq x\rho_n+A_n/\sig_n)=F_n(x\rho_n+A_n/\sig_n).
$$
Notice that $\rho_n\to1$.
As a consequence, if we define 
$$
\tilde \Phi_{m,n}(x)=\check\Phi_{m,n}(\rho_n x+A_n/\sig_n)=\Phi(\rho_n x+A_n/\sig_n)-\varphi(\rho_n x+A_n\sig_n)\sum_{j=1}^{m-2}n^{-j/2}\tilde H_{j}(\rho_n x+A_n/\sig_n)
$$
then all the results stated in Section \ref{Main} hold true with  $\textbf{W}_n$ instead of $W_n$ and with $\tilde \Phi_{m,n}$ instead of $\Phi_{m,n}$. To complete the proof of Proposition \ref{pass2}  we need to prove the following result:

\begin{lemma}\label{L3}
There are polynomials $\bar H_j$ whose coefficients are rational functions of  $p_k$ and $q_k$ for $k\leq 3j$ so that with 
$$
\bar\Phi_{n,m}(x)=\Phi(x)-\varphi(x)\sum_{j=1}^{m-2}n^{-j/2}\bar H_{j}(x)
$$
uniformly in $x$ we have 
\begin{equation}\label{Goal}
\tilde\Phi_{n,m}(x)=\bar\Phi_{n,m}(x)+R_n(x)e^{-x^2/4}O(n^{-m/2})
\end{equation}
where $R_n$ is a polynomial with bounded coefficients whose degree does not depend on $n$.
\end{lemma}

\begin{proof}[Proof of Lemma \ref{L3}]
Recall that
\begin{equation}\label{tilde}
\tilde \Phi_{m,n}(x)=\Phi(\rho_n x+A_n/\sig_n)-\varphi(\rho_n x+A_n/\sig_n)\sum_{j=1}^{m-2}n^{-j/2}\tilde H_{j}(\rho_n x+A_n/\sig_n).
\end{equation}

First, let us suppose that $|x|\geq n^\ve$ for some  fixed $\ve\in(0,1/2)$. Then, 
$$
\max\left(|\bar\Phi_{m,n}(x)|, |\tilde\Phi_{m,n}(x)|\right)\leq C_m(1+|x|^{u_m})e^{-x^2/2}
$$
for any choice of polynomials $\bar H_j$
for some $u_m$ which depends only on $m$ and the polynomials $\bar H_j$. Thus, when $|x|\geq n^{\ve}$ then for any choice of polynomials $\bar H_j$, both $\tilde\Phi_{m,n}(x)$ and $\bar\Phi_{m,n}(x)$ are of order $|x|^{u_m}e^{-x^2/4}e^{-n^{2\ve}/4}$ for some $u_m$ which depends only on $m$ and the polynomials. Since 
$$
e^{-n^{2\ve}/4}=O(n^{-m/2})
$$
we see that it is enough to find polynomials $\bar H_j$ which, given $n$ large enough, satisfy the desired properties for points $x$ so that $|x|\leq n^\ve$, where again $\ve$ is a fixed constant which can be chosen to be arbitrarily small.

Let us fix some small $\ve>0$ and let $n$ and $x$ be so that $|x|\leq n^\ve$. The idea now is to expand all the functions of $\rho_n x+A_n/\sig_n$ from the definition of $\tilde \Phi_{m,n}$ in powers of $n^{-1/2}$. Let us start with $\Phi(\rho_n x+A_n/\sig_n)$.
Using Assumption \ref{AssStat}, with $\sig=\sqrt{p_2}>0$ we can write
$$
\eta=\eta_{n,x}:=x(\rho_n-1)+A_n/\sig_n=x\left(\frac{p_2 n-\sig_n^2}{\sig_n(\sig_n+\sig\sqrt n)}\right)+A_n/\sig_n
$$
$$
=
\frac{-x q_2}{\sig_n(\sig_n+\sig\sqrt n)}+A_n/\sig_n+O(\del^n)=\frac{-x q_2}{\sig_n(\sig_n+\sig\sqrt n)}+c/\sig_n+O(\del^n):=\eta_1+O(\del^n)
$$
where the constant $c$ is the one satisfying $A_n=c+O(\del^n)$.
Then, using that $|x|\leq n^{\ve}$, $\sig_n^2\asymp p_2 n$ and $A_n=O(1)$ we see that $|\eta|\leq Cn^{-1/2}$.
Using now the formula for the Taylor remainder of $\Phi$ around $x$, and taking into account that $\rho_n=1+o(1)$ and $A_n=O(1)$ we see that for all $s\in\bbN$,
$$
\Phi(x\rho_n+A_n/\sig_n)=\Phi(x+\eta)=\Phi(x)+\sum_{j=1}^{s}\frac{\Phi^{(j)}(x)}{j!}\eta^j+r_n(x)
$$
where 
$$
|r_n(x)|\leq C_s(1+|x|^{a_s})e^{-x^2/4}n^{-\frac12(s+1)}
$$
and $a_s$  and $C_s$ depend  on $s$ but not on $x$ or $n$. Now, by expanding $\eta^j$ in powers of $\eta_1$,  absorbing the powers of the term $O(\del^n)$ times $\frac{\Phi^{(j)}(x)}{j!}$ in $r_n(x)$ and using that for $j\geq2$ we have
$$
\Phi^{(j)}(x)=\varphi^{(j-1)}(x)=(-1)^{j-1}\varphi(x)H_{j-1}(x)
$$
we see that, possibly with different constants $a_s,C_s$, we have
$$
\Phi(x\rho_n+A_n/\sig_n)=\Phi(x)-\varphi(x)\sum_{j=1}^{s}\frac{(-1)^{j-1}}{j!}(\zeta_nx+c/\sig_n)^jH_{j-1}(x)+r_n(x)
$$
where with $\sig=\sqrt{p_2}$,
$$
\zeta_n=\frac{-q_2}{\sig_n(\sig_n+\sig\sqrt n)}.
$$
By taking $s=s_m$ so that $\frac12(s+1)>(m-1)/2$,  expanding $(\zeta_nx+c/\sig_n)^j$ using the Binomial formula and expanding
$\zeta_n^k, k\leq s$ in powers of $n^{-1/2}$ (using that $\sig_n^2=np_2+q_2+O(\del^n)$ and the Taylor expansions of $g_1(y)=\sqrt{1+y}$ and $g_2(y)=\frac{1}{1+y}$  around the origin) we see that 
\begin{equation}\label{One}
\left|\Phi(x\rho_n+A_n/\sig_n)-\left(\Phi(x)-\varphi(x)\sum_{j=1}^{m-2}n^{-j/2}E_j(x)\right)\right|\leq C_m(1+|x|^{v_m})e^{-x^2/4}n^{-(m-1)/2}
\end{equation}
where $v_m$ and $C_m$ depend only on $m$, and $E_j$ are polynomials whose coefficients depend only on  $\sig=\sqrt{p_2}$, $p_k$ and $q_k$ (the exact formula for $E_r$ can be recovered from the latter two expansions).

Similar arguments show that we can expand $\varphi(x\rho_n+A_n\sig_n)$, namely there are polynomials $U_j$
 so that 
 \begin{equation}\label{Two}
\left|\varphi(x\rho_n+A_n/\sig_n)-\varphi(x)\sum_{j=1}^{m-2}n^{-j/2}U_j(x)\right|\leq C_m'(1+|x|^{b_m})e^{-x^2/4}n^{-(m-1)/2}
\end{equation}
with $b_m$ and $C_m'$ depending only on $m$. Thus, so far we have managed to replace the terms 
$\Phi(\rho_n x+A_n/\sig_n)$ and $\varphi(\rho_n+A_n/\sig_n)$ by $\Phi(x)$ and $\varphi(x)$ times polynomials of above form, and it remains to ``handle" the terms $\tilde H_r(\rho_n x+A_n/\sig_n)$. 

Let us fix some $j$. Then, using the same notations as above, since $\tilde H_j$ is a polynomial we have 
$$
\tilde H_j(\rho_n x+A_n/\sig_n)=\tilde H_j(x+\eta)=\sum_{l=0}^{w_j}\left(\sum_{k=l}^{w_j}a_k\binom{k}{j}\eta^{k-j}\right)x^{l}
$$
where $w_j$ is the degree of $\tilde H_j$ and $a_k=a_k(j)$ are it's coefficients. Now we can further expand $\eta^{k-l}$ in powers of $n^{-1/2}$ and disregard the $O(\del^n)$ terms. Notice that since $|x|^{w_j}\leq n^{\ve w_j}$ then by taking $\ve$ small enough $x^l$ times the error term in the above approximation would still be $O(n^{-(m-1)})$.
 By possibly omitting terms of order $n^{-s/2}$ for $s>m-1$, we conclude that there are polynomials $V_{j,r}$ so that when when $|x|\leq n^{\ve}$ and $\ve$ is small enough then for all $j\leq m-2$ we have
\begin{equation}\label{Three}
\left|n^{-j/2}\tilde H_j(\rho_n x+A_n/\sig_n)-\sum_{r=j}^{m-2}n^{-r/2}V_{j,r}(x)\right|\leq Cn^{-(m-1)/2}
\end{equation}
where $C$ is some constant.

Plugging in \eqref{One}, \eqref{Two} and \eqref{Three} inside \eqref{tilde},  and possibly  disregarding terms of order $n^{-s/2}$ for $s>m-1$, we arrive at \eqref{Goal} with some polynomials $\bar H_j$ whose coefficients can be computed by keeping track of all the above expansions (one can first consider expansions in powers of $n^{-1/2}$ smaller than $m$ in each one of the above three expansions, and then disregard terms of order $O(n^{-\frac12(m-1)})$  which come from multiplying all three expansions).
\end{proof}


\section{Examples with  some stationarity/homogeneity}\label{Stat}

\subsection{General scheme}
In this section we describe a functional analytic framework that in the next section will be ``verified" for several classes of stationary sequences. 

\begin{assumption}\label{GenAss}
There are analytic functions $\Pi(z),U(z)$ and $\del_n(z)$ of a complex variable $z$ which are defined on a complex neighborhood of $0$ so that:  
\begin{enumerate}
\item $\Pi(0)=0$, $\Pi''(0)>0$ and $U(0)=1$;
\vskip0.2cm
\item 
$$
|\del_n(z)|\leq C\del^n
$$
for some $\del\in(0,1)$ and $C>0$;
\vskip0.2cm
\item We have
\begin{equation}\label{Stat assmp}
\bbE[e^{zS_n}]=e^{n\Pi(z)}(U(z)+\del_n(z))
\end{equation}
\end{enumerate}
\end{assumption}
We note that since $\del_n(0)=0$ and $\del_n(z)$ is uniformly bounded it follows that $|\del_n(z)|\leq C'|z|\del^n$, and so, by possibly considering a smaller neighborhood of $0$ we can always assume that $|\del_n(z)|\leq \ve \del^n$ for an arbitrary small $\ve$.

\begin{lemma}\label{Approx Lemma}
For all $m$, all the conditions specified in Assumptions \ref{GrowAssum} and \ref{AssStat} are met under Assumption \ref{GenAss}.
 Moreover, 
 \begin{equation}\label{Relat0}
 \sig_n^2=n\Pi''(0)+H''(0)+O(\del^n).
\end{equation}
Furthermore, for all $j\geq 3$ we have
 \begin{equation}\label{Relat1}
  \gamma_j(S_n)=n\Pi^{(j)}(0)+H^{(j)}(0)+O(\del^n).
 \end{equation}
 In other words, Assumption \ref{AssStat} is in force with $p_k=\Pi^{(k)}(0)$ and $q_k=H^{(k)}(0)$.

\end{lemma}
\begin{proof}
By taking the logarithm of both sides of \eqref{Stat assmp} we see that
$$
\Gamma_n(z):=\ln\bbE[e^{zS_n}]=n\Pi(z)+H(z)+r_n(z)
$$
where $r_{n}(z)=O(|z|\del^n)$ and $H(z)=\ln U(z)$ are both analytic functions.
Now \eqref{Relat0} and \eqref{Relat1} follow by differentiating $\Gamma_n(z)$, using the Cauchy integral formula and plugging in $z=0$. 
Note also that the same argument yields that for $j\geq3$, for all $|t|$ small enough we have
 $$
 \Lambda_n^{(j)}(t)=\sig_n^{-j}\Gamma_n^{(j)}(it/\sig_n)=\sig_n^{-j}\left(n\Pi^{(j)}(it/\sig_n)+H^{(j)}(it/\sig_n)\right)+O(\del^n)
 $$
 and since $\sig_n^2\asymp \Pi''(0)n$ we see that
  Assumption \ref{GrowAssum} holds true for every $m$. 
 \end{proof} 
 
\subsection{A functional analytic setup}\label{Sec5.2}
 Let $(B,\|\cdot\|)$ be a Banach space of functions on some measurable space $\cX$ so that the constant functions are in $B$ and  $\|g\|\geq\|g\|_{L^1(\mu)}$ for some probability measure $\mu$ which, when viewed as  linear functional, belongs to the dual $B^*$ of $B$.
   We also assume that $\|fg\|\leq C\|f\|\|g\|$ for some $C>0$ and all $f,g\in B$. 
 In addition, we  suppose that there is an operator $\cL$ acting on $B$  so that $\cL^*\mu=\mu$ and a sequence of $\cX$-valued random variables $\{U_j:j\geq1\}$ so that  for all $f_1,...,f_n\in B$ we have
 \begin{equation}\label{ExAss}
  \bbE\left[\prod_{j}f_j(U_j)\right]=\mu\left(\cL_{f_n}\circ\cdots\circ \cL_{f_2}\circ \cL_{f_1}h_0\right)
 \end{equation}
where  $h_0\in B$ is positive and satisfies $\mu(h_0)=1$, $\cL h_0=h_0$  and
 $$
 \cL_{f_j}(g):=\cL(f_j g).
 $$
 We then take $f\in B$ so that $\mu(f)=0$ and set $S_n=\sum_{j=1}^n X_j$, for $X_j=f_j(U_j)$. For each complex number $z\in\bbC$ we also set 
 $$
 \cL_z(g)=\cL_{e^{zf}}(g)=\cL(e^{z f}g).
 $$
 Let us also assume that\footnote{It will follow from the assumptions of Proposition \ref{PrP} that the limit $\sig^2=\Pi''(0)$ always exists, and so the main additional assumption here concerns the positivity of $\sig^2$. In our applications, a characterization for the latter positivity are well known (and will be recalled in the next sections).} 
 $$
p_2=\sig^2=\lim_{n\to\infty}\frac1n\text{Var}(S_n)>0.
 $$
We have the following.
 \begin{proposition}\label{PrP}
In the above circumstances,
 Assumption \ref{GenAss} is in force for all $m$ if
 $z\to\cL_z$ is well defined analytic in $z$ in a complex neighborhood of the origin, the operator $\cL$ is quasi-compact, the spectral radius of $\cL$ is $1$, and up to a multiplicative constant, the function $h_0$ is the unique eigenfunction corresponding to an eigenvalue of modules one.  
 \end{proposition}
 
\begin{proof}
By applying an analytic perturbation theorem we get that there is a number $r>0$ so that for every $z\in\bbC$ with $|z|<r$ there are analytic functions of $z\to \la(z)\in\bbC\setminus\{0\}$, $z\in h^{(z)}\in B$ and $z\to\nu^{(z)}\in B^*$ so that 
$$
\left\|\la(z)^{-n}\cL_z^n-\nu^{(z)}\otimes h^{(z)}\right\|\leq C\del^n
$$
where $C>0$, $\del\in(0,1)$ and the operator $\nu\otimes h$ is given by $g\to\nu(g)h$. Moreover, $\la(0)=1$, $\nu^{(0)}=\mu$ and $h^{(0)}=h_0$.
Thus,  
$$
\bbE[e^{zS_n}]=\mu(\cL_z^nh_0)=e^{n\ln\la(z)}\left(U(z)+\del_n(z)\right)
$$
where $U(z)=\nu^{(z)}(h_0)\mu(h^{(z)})$ (which indeed takes the value $1$ at $z=0$) and $\del_n(z)=O(\del^n)$ for some $\del\in(0,1)$. 
\end{proof}

  
  \begin{remark}
  In order to verify Assumption \ref{GrowAssum} with  a given $m$, it is enough to assume that the operators $t\to \cL_{it}$ are  $C^{m+1}$ in a (real) neighborhood of the origin. Then, under the same conditions on $\cL$ like in Assumption \ref{PrP} (see \cite{HH}) we get that 
  \begin{equation}\label{lpp}
    \cL_{it}^n=\la(t)^n\nu^{(it)}\otimes h^{(it)}+N^n(t)=\la(t)^n\left(\nu^{(it)}\otimes h^{(it)}+\tilde N^n(t)\right)
  \end{equation}
  where  $\tilde N(t)=N(t)/\la(t)$ and $N(t)$ satisfies that
$\|N^n(t)\|\leq C\del^n$ for some $\del\in(0,1)$ (all the expressions are $C^{m+1}$ in $t$). Now, since $
\la(0)=1$, if $|t|$ is small enough we have  $\|\tilde N^n(t)\|\leq \del_1^n$  for some $\del_1\in(0,1)$. 
Since all the first $m+1$ derivatives of $\tilde N(t)^n$ decay exponentially fast to $0$ as $n\to\infty$ we can verify Assumption \ref{GrowAssum} for the given $m$ by differentiating the logarithms of both sides of \eqref{lpp}.
However, in the applications we have in mind the operators $z\to\cL_z$ will already be analytic.
   \end{remark}
 
  Now, let us provide abstract conditions that yield estimates on the integral appearing in \ref{DerAss}.
\begin{proposition}\label{Der Lem}
Let $m\geq 3$.
Suppose that for some $c$ domain of the form $\{|t|\in [c, Bn^{(m-3)/2}]\}$, $c,B>0$ we have 
\begin{equation}\label{Derrr}
\|\cL_{it}^n\|\leq A(1+C|t|)e^{-c_0 n |t|^{-\zeta}}
\end{equation}
for some constants  $A,C>0$  and $0\leq \zeta<\frac{2}{m-3}$ (where when $m=3$ we set $\frac{2}{m-3}=\infty$). Then the estimate in Assumption \ref{DerAss} holds true with these $m,c$ and some $B'>0$.
\end{proposition}

\begin{proof}
As explained in Remark \ref{Rem psi}, it is enough to verify the condition \eqref{psi cond}. 
First,
\begin{equation}\label{Begg}
\frac{d^m}{dt^m}\, \bbE[e^{it S_n}]=i^m\bbE[(S_n)^me^{itS_n}]=i^m\sum_{1\leq \ell_1,...,\ell_m\leq n}C_{\bar\ell,m}\bbE[X(\bar\ell)e^{it S_n}]
\end{equation}
where 
$$
C_{\bar\ell,m}=\frac{ m!}{\ell_1!\cdots\ell_{n}!}
$$
and for a given $\bar\ell=(\ell_1,...,\ell_m)$, if $\ell_{s_q}, q\leq m(\bar\ell)$ are the nonzero ones among $\ell_s$ with $s_j<s_{j+1}$ and $ m(\ell)\leq m$ we have 
$$
X(\bar\ell)=\prod_{q=1}^{m(\ell)}X_{s_q}^{\ell_{s_q}}.
$$
Now, let us write $\{1,2,...,n\}\setminus\{s_1,...,s_{m(\bar\ell)}\}$ as a union of disjoint intervals 
$I_1,....,I_d$ so that $I_i$ is to the left of $I_{i-1}$. Let $E_j=\{u_j,u_j+1,...,v_j\}$ be the gap between $I_j$ and $I_{j+1}$. Then $E_j$ is a union of $s_u$'s and hence its length is bounded by $m$. Moreover, the number of gaps is also bounded by $m$ (so $d\leq m$). 
 Denote $e_r=e_{r,t}=e^{it f}f^{r}$, and
let the operator $L_{j,t}$ be given by 
$$
L_{j,t}=\cL_{e_{s_{v_j}}}\circ\cdots \circ\cL_{e_{s_{u_j+1}}}\circ\cL_{e_{s_{u_j}}}.
$$
Then, using \eqref{ExAss},
$$
\bbE[X(\bar\ell)e^{it S_n}]=\mu\left(\cL_{it}^{|I_d|}\circ\cdots \circ L_{2,t}\circ \cL_{it}^{|I_2|}\circ L_{j,1}\circ \cL_{it}^{|I_1|}\textbf{1}\right)
$$
and hence
$$
\left|\bbE[X(\bar\ell)e^{it S_n}]\right|\leq C'\left(\prod_j\|L_{j,t}\|\right)\left(\prod_{a}\|\cL_{it}^{|I_a|}\|\right).
$$
Next,
notice that $\cL_{e_{s_{a}}}(g)=\cL_{it}(f^{s_a}g)$. Thus, since $s_a\leq m$ by using \eqref{Derrr} with $n=1$ we see that 
$$
\|\cL_{e_{s_{a}}}\|\leq C_0\|\cL_{it}\|\leq AC_0(1+C|t|)
$$
for some $C_0$ which depends only on $f$ and $k$ and the norm $\|\cdot\|$. We thus conclude that 
$$
\prod_j\|L_{j,t}\|\leq \tilde C''(1+C|t|)^{m}.
$$

In order to bound $\prod_{a}\|\cL_{it}^{|I_a|}\|$ we use \eqref{Derrr} and 
notice that at least one of the lengths $|I_a|$ of the $I_a$'s is at least of order $[n/2m]$. We conclude that 
$$
\left|\bbE[X(\bar\ell)e^{it S_n}]\right|\leq A'(1+C|t|)^{2m}e^{-c_0 n |t|^{-\zeta}/m}.
$$ 
By summing over all multi-indexes $\bar\ell$ (and using \eqref{Begg}) we get that 
$$
\left|\frac{d^m}{dt^m}\, \bbE[e^{it S_n}]\right|\leq  A_k''(1+C|t|)^{2m}n^{m}e^{-c_0 n |t|^{-\zeta}/m}.
$$
Hence, on the domain $|t|\in [c, Bn^{(m-3)/2}]$ we see that 
$$
\left|\frac{d^m}{dt^m}\, \bbE[e^{it S_n}]\right||t|^{-1}\leq A_m''|t|^{2m-1}e^{-\ve_m n^{1-(m-3)\zeta/2}}, \ve_m>0
$$
which is enough for \eqref{psi cond} to hold  (recall that $1-(m-3)\zeta/2>0$ and $\sig_n\asymp \sqrt{p_2}n$).

\end{proof}

\begin{remark}\label{NonAR Rem}
(i) When $m=3$ then the conditions of Lemma \ref{Der Lem} hold true for every $c$ and $B$ when the spectral radius of $\cL_{it}, t\not=0$ is smaller than $1$. Indeed, in this case for every compact subset $K$ of $\bbR\setminus\{0\}$ we have (see \cite{HH}),
$$
\sup_{t\in K}\|\cL_{it}^n\|\leq C(K)e^{-c(K)n},\, c(K)>0
$$
This is usually the case when  the function $f$ is non-arithmetic (aperiodic) in an appropriate sense, see \cite{GH,HH}.

(ii) $m>3$, we  note that
if the assumptions of Lemma \ref{Der Lem} do not hold true for every $c$ small enough then 
in order to verify Assumption \ref{DerAss} we also need to estimate the integrals over $[a\sig_n, c\sig_n]$ with small $a$'s. 
This is also requiring that the spectral radius of $\cL_{it}$ for $t\not=0$ is smaller than $1$.
\end{remark} 
\subsection{Applications to homogeneous elliptic Markov chains}
Let $(\xi_j)$ be an homogeneous (not necessarily stationary) Markov chain taking values on some measurable space $\cX$. 
Let $R$ be the corresponding Markov operator, that is $R$  maps a bounded function $g$ on $\cX$ to the function $Rg$ on $\cX$ given by
$$
Rg(x):=\bbE[g(\xi_{1})|\xi_0=x].
$$ 
Let us assume that $R$ has a continuous action on a Banach space $(B,\|\cdot\|)$ of function on $\cX$ with a norm $\|\cdot\|$ satisfying $\|g\|\geq \sup|g|$, and that for every function $f\in B$ and a complex number $z$ we have $e^{zf}\in B$ and that the function $z\to e^{zf}$ is analytic. Moreover, let us assume that $Rg$ is quasi compact and that the constant function $\textbf{1}$ is the unique eigenfunction with eigenvalue of modulus $1$. Namley, there is a unique stationary measure $\mu$   for the chain and constants $\del\in(0,1)$ and $C>0$ so that for every function $g\in B$ we have
$$
\|R^n g-\mu(g)\|\leq C\|g\|\del^n.
$$ 

\begin{example}
Suppose that the classical Doeblin condition holds true:  there exists a probability measure $\nu$ on $\cX$ and a constant $c>0$ so that for every measurable set $\Gamma\subset\cX$ we have
 $$
 \bbP(\xi_{n_0}\in\Gamma|\xi_0=x)\geq c\nu(\Gamma).
 $$
 Then the Markov chain is geometrically egrodic, that is $R$ satisfies the above properties with the norm $\|g\|=\sup|g|$. 
\end{example}

Next,  let $f\in B$  and set 
$$
S_n=\sum_{j=1}^n(f(\xi_j)-\bbE[f(\xi_j)]).
$$
 For a complex parameter $z$, let 
$$
R_z(g)=R_z(e^{zf}g).
$$

In \cite{NonU BE} we showed that Assumption  \ref{GenAss} is in force in the above circumstances. Therefore we have the following result.
\begin{theorem}
Let
$A_n=O(1)$ and either  $B_n=\sig\sqrt n+O(1)$ or $B_n=\sig_n+O(1)$, where $\sig=\lim_{n\to\infty}\frac1n\text{Var}(S_n)$.
Then  the Berry Esseen theorem in the transport distances  of all orders holds (Theorem \ref{BE1} for all $m$).  Moreover, the coupling described in Corollary \ref{ASIP} exist for all $n$ and $m$.
\end{theorem}
\begin{remark}
We would like also to refer to the second example in Section \ref{Add} for related types of Markov chains for which Assumption  \ref{GenAss} (and hence the above theorem) hold true.
\end{remark}

\subsubsection{Expansions of order $1$}
As mentioned in Remark \ref{NonAR Rem}, the conditions of Lemma \ref{Der Lem} are always satisfied with $m=3$ when $f$ is non-arithmetic. We thus conclude that in this case $W_n$ obeys all the types of Edgeworth expansions presented in Section \ref{Main} when $m=3$ (i.e. all the first order expansions hold true).

\subsubsection{Higher order expansions}
In order to derive the non-uniform Edgeworth expansions  (and their consequences) or order $r=m-2>1$  we need to restrict ourselves to a more specific framework. Let us assume  that $\cX$ is a Riemannian manifold and that 
$$
Rg(x)=\int_{\cX}p(x,y)g(y)dy
$$
for some positive measurable function $p(x,y)$ which is bounded and bounded away from the origin,
where $dy$ is the volume measure.

\begin{theorem}
If $f$ is H\"older continuous with exponent $\al\in(0,1]$ then Assumption \ref{DerAss} holds true when  $m-2<\frac{1+\al}{1-\al}$, where when $\al=1$ we set $\frac{1+\al}{1-\al}=\infty$. 
\end{theorem}
\begin{proof}
Combining  \cite[Lemma 4.3]{DS} with \cite[Lemma 35]{DolgHaf} (restricted to the homogeneous case) we have
$$
\|R_{it}^4\|\leq e^{-cd^2(t)},\,\,c>0
$$
where for every $\del>0$ on the domain $\{|t|\in [\del,\infty)\}$  we have  $d^2(t)\asymp |t|^{1-1/\al}$.
Hence, on every domain of the form $\{|t|\in[\del,\infty)\}$ we have
$$
\|R_{it}^n\|\leq e^{-c_1|t|^{1-1/\al}n},\,\,c_1>0
$$
This shows that the conditions of Proposition \ref{Der Lem} hold true with every $c>0$, and hence Assumption \ref{DerAss} is in force.
\end{proof}
\begin{remark}
The condition $m-2<\frac{1+\al}{1-\al}$ is optimal even for the uniform expansions (even in the iid case), see \cite{DolgHaf} (note that the role of $m-1$ there is similar to the role of $r$ in  \cite{DolgHaf}).
\end{remark}

\subsection{Application to expanding or hyperbolic dynamical systems}\label{SecDS}
Let $(X,\cB,\mu)$ be a probability space and let $f:\cX\to\bbR$.  Let $T:\cX\to\cX$ be a measurable map  and set  
$$
U_j=T^j X_0
$$
where $X_0$ is a random element in $\cX$ whose distribution is $\mu$. Let $B$ be a Banach space of measurable functions on $\cX$ with the properties described in Section \ref{Sec5.2}.
Then the conditions of Proposition \ref{PrP} hold true for a variety of (non-uniformly) expanding maps and measures, some of which listed below:

\begin{itemize}
\item $(X,\cB,T)$ is a topologically mixing one sided subshift of finite type, $\mu$ is a Gibbs measure corresponding to some H\"older continuous potential with exponent $\al$ and  $B$ is the space of H\"older continuous functions, equipped with the norm 
$$
\|g\|=\sup|g|+v_\al(g)
$$
where $v_\al(g)$ is the H\"older constant of $g$ corresponding to the exponent $\al$.
 This has applications to Anosov maps $T$ and invariant measure $\mu$ which are equivalent to the volume measure, and we refer to \cite{Bowen} for the details. 

\item $(X,\cB,T)$ is a locally expanding dynamical system on some compact manifold $X$, $\mu$ is the normalized volume measure and $B$ is the space of functions with bounded variation (see \cite{KelLiv}). 

\item $(X,\cB,T)$ is an aperioic non-invertible Young tower with exponential tails, $B$ is the space (weighted) Lipschitz continuous functions and $\mu$ is the lifted volume measure. This has a variety of applications to partially hyperbolic or expanding maps, and we refer to \cite{Y1} for the details (see also \cite{MelNic} for a related setup). 
\end{itemize}

\subsection{Expansions of order $1$}
Let $f$ be aperiodic in the sense of \cite{GH,HH} (see Remark \ref{NonAR Rem}). Then the spectral radius of $\cL_{it}$ is smaller than $1$ for every $t\not=0$. Moreover,  $f$ is not a coboundary with respect to the map $T$ (i.e. $f$ does not have the form $f=r-r\circ T$). As mentioned in Remark \ref{NonAR Rem} this is already sufficient for all of the Edgeworth expansions stated in this paper to hold with $m=3$ (i.e. for the first order Edgeworth expansions to hold).

\subsubsection{Expansions of all orders}
Let $f$ be aperiodic. 
As mentioned in Remark \ref{NonAR Rem}, in order to verify Assumption \ref{DerAss} when $m>3$ we need to show that  the conditions of  \ref{Der Lem} hold for some $c>0$ and all $B$. 
Listed below are types of maps  that satisfy these conditions for all  functions $f\in B$ which are not a coboundary with respect to the map $T$.

\begin{itemize}
\item The one dimensional expanding maps with discontinuities as in \cite{BP}  (for all $m$ and functions $f$ with bounded variation).
\vskip0.1cm
\item The multidimesional expanding  maps as in \cite{DolgIsr} (for all $m$ and $C^1$-functions $f$, see \cite[Section 3]{DolgIsr}).
\end{itemize}

\subsection{Application to products of random matrices in a ``neighborhood" of a given random matrix}\label{SecMat}
Let $V$ be a $d$ dimensional vector space for some $d>1$. Let us fix some scalar product on $V$ and denote by $\|\cdot\|$ the corresponding norm. Let us denote by $X=\bbP V$ the projective space space on $V$, equipped with a suitable Remannian distance $d(\cdot,\cdot)$ (see \cite[Chapter II]{BoLac}). Given $x\in V$ and a sequence $(g_n)$ of iid random variables which takes values on $G=GL(V)$, we define 
$$
S_n(x)=\log\frac{\|g_n\cdots g_1 x\|}{\|x\|}.
$$
Then, as usual, $S_n(x)$ can be represented as the ergodic sum of $\bar\phi(g,x)=\frac{\log \|gx\|}{\|x\|}$ (see \cite{HH}). 

Let $\mu$ be the common distribution of all $g_n$. Let us consider the following assumptions.
\begin{assumption}\label{Ass1 mat}[Exponential moments]
For some $\del>0$ we have 
$$
\int_G\max\left(\|g\|,\|g^{-1}\|\right)^\del d\mu(g)<\infty.
$$
\end{assumption}
and

\begin{assumption}\label{Ass2 mat}[Strongly irreducibility and proximal elements]
The semi group generated by the support $\Gamma_\mu$  of $\mu$ has the property that there is no finite union of proper subspaces which is $\Gamma_\mu$-invariant. Moreover, there exits $g\in \Gamma_\mu$ which has a simple dominant eigenvalue. 
\end{assumption}

As will be explained in what follows, using the arguments of  \cite{Guiv}, under the above two assumptions the following limits exist
$$
\la_1=\lim_{n\to\infty}\frac1n\bbE[\log\|g_n\cdots g_1\|]
$$
and 
$$
\sig^2=\lim_{n\to\infty}\frac1n\bbE[\left(\log\|g_n\cdots g_1\|-n\la_1\right)^2].
$$
Recall also that for every $x\in\bbP V$ we have 
\begin{equation}\label{ee}
\la_1=\lim_{n\to\infty}\frac1n S_n(x), \text{a.s.}
\end{equation}
and that the above limit is also in $L^1$, uniformly in $x$. In particular,
$$
\la_1=\lim_{n\to\infty}\frac1n\bbE[S_n(x)].
$$
Moreover, for every $x\in\bbP V$ we have
\begin{equation}\label{vv}
\sig^2=\lim_{n\to\infty}\frac1n\bbE[\left(S_n(x)-n\la_1\right)^2].
\end{equation}
The latter  actually also follows from the CLT for $n^{-1/2}(S_n(x)-\la_1 n)$ (see \cite{FP}) together with the existence of the limit $\lim_{n\to\infty}\frac1n\bbE[\left(S_n(x)-n\la_1\right)^2]$ which will be proven in the proceeding arguments. 
Finally, note also that, for more general potentials $\sig^2=0$ if and only if the potential $\bar\phi(g,x)$ admits an appropriate coboundary representation. For the above choice of $\bar\phi$ we have that $\sig^2>0$ (see \cite{HH, Guiv}).

The following results follows from arguments in the proof of \cite[Proposition 28]{NonU BE} and the fact that the function $\bar\phi(g,x)=\frac{\log \|gx\|}{\|x\|}$ is irreducible (see \cite{HH,Guiv}). 
\begin{proposition}\label{MatProp}
Under Assumptions \ref{Ass1 mat} and \ref{Ass2 mat} Assumption \ref{GrowAssum} is valid for all $m$, Assumption \ref{DerAss} is valid  with $m=3$ and Assumption  \ref{GenAss} (and so \ref{AssStat}) is in force.  Moreover,  we have $B_n:=\sig \sqrt n=\sig_n+O(1)$ and $A_n:=\la_1 n-\bbE[S_n]=O(1)$. Thus the sequence
$\frac{S_n(x)-n\la_1}{\sig\sqrt n}$  obeys the Berry-Esseen Theorem as in Theorem \ref{BE1}. Moreover, $S_n-\la_1 n$ satisfies the results of Corollary \ref{ASIP} with all $m$ (with $S_n-\la_1 n$ instead of $S_n$). In addition, the non-uniform Edgeworth expansion of order $1$ holds true (i.e. Theorem \ref{Edge2} with $m=3$), as well as Theorem \ref{Edge1} and Corollary \ref{Cor} hold true with $m=3$ (and polynomials which do not depend on $n$).
\end{proposition}

Next, in order to obtain Edgeworth expansions of order $r>1$ let us consider the following condition.
\begin{assumption}\label{Ass3 mat}
The support of $\mu$ is Zariski dense in a connected algebraic subgroup of $G_L(V)$.
\end{assumption}

\begin{proposition}
Under assumption \ref{Ass3 mat}, the integrability condition specified in Assumption \ref{DerAss} is met with all $m$. As a consequence, the sequence $\frac{S_n-\la_1 n}{\sig \sqrt n}$ obeys non-uniform Edgeworth expansions of all orders $m$ with $B_n=\sqrt n \sig$ and stationary correction terms. Thus the corresponding (stationary) expansions in the transport distances as in Theorem \ref{Edge1} and expansions of the expectations as in \ref{Cor} (ii) hold true for all $m$. 
\end{proposition}
\begin{proof}
First, since $\bar\phi$ is aperiodic for all $0<\del<K$ we have that $\sup_{\del\leq |t|\leq K}\|\cL_{it}^n\|_{\ve}$ decays to $0$ at exponential rate (see \cite{Guiv}). Now, by Assumption \ref{Ass3 mat} (see \cite[Theorem 4.19]{38.}) there exists $K>0$ so that if $\ve$ is small enough then $\|\cL_{is}^n\|_\ve\leq C|t|^{2\ve}e^{-cn}$ for some $C,c>0$ and all $t$ so that $|t|\geq K$. Since 
$$
\left|\bbE[e^{it (S_n(x)-\bbE[S_n(x)])}]\right|=\left|\bbE[e^{it (S_n(x)-n\la_1)}]\right|=|\cL_{it}^{n}\textbf{1}(x)|\leq \|\cL_{it}\|_\ve 
$$ 
(for every $\ve$),
we see that the conditions of Lemma \ref{Der Lem} hold with all $m$.
\end{proof}


\subsection{Applications to skew products (annealed limit theorems): a brief discussion}
In \cite{ANM} it was shown that for some classes of random iid expanding the corresponding skew product is controlled by means of an (``averaged")  quasi-compact operator $\cA$. That is, if the underlying maps $T_\om$ map a space $\cX$ to itself and the corresponding skew product is denoted by $T$, then for every function $\psi$ with bounded variation we have 
$$
\bbE[e^{zS_n\psi}]=\mu(\cA_z^n\textbf{1})
$$ 
where $S_n\psi=\sum_{j=0}^{n-1}\psi\circ T^j$ and $\cA_z(g)=\cA(ge^{z\psi})$. Since $\cA_z$ is analytic in $z$, arguing as in the previous sections we see that Assumption \ref{GrowAssum} holds true with all $m$.

\section{Examples: Nonstationary dynamical systems and Markov chains}\label{NonStat}

\subsection{Inhomogenous uniformly elliptic Markov chains}\label{MC IH}
Let $(\cX_i,\cF_i),\,i\geq1$ be a sequence of measurable spaces.
For each $i$,  let $R_i(x,dy),\,x\in\cX_i$ be a measurable family of (transition) probability measures on $\cX_{i+1}$. Let $\mu_1$ be any probability measure on $\cX_1$, and let $X_1$ be an $\cX_1$-valued random variable with distribution $\mu_1$. Let $\{X_j\}$ be the Markov 
 starting from
 $X_1$ with the transition probabilities
\[
\bbP(X_{j+1}\in A|X_{j}=x)=R_{j}(x,A),
\] 
where $x\in\cX_j$ and $A\subset\cX_{j+1}$ is a measurable set.
Each $R_j$ also gives rise to a transition operator given by 
\[
R_j g(x)=\bbE[g(X_{j+1})|X_j=x]=\int g(y)R_j(x,dy)
\] 
which maps an integrable function $g$ on $\cX_{j+1}$ to an integrbale function on $\cX_j$ (the integrability is with respect to the laws of $X_{j+1}$ and $X_j$, respectively). We assume here that 
there are probability measures $\fm_j$,
$j>1$ on $\cX_j$ and families of transition probabilities $p_j(x,y)$ 
so that 
\[
R_j g(x)=\int g(y)p_j(x,y)d\fm_{j+1}(y).
\]
Moreover, there exists $\ve_0>0$ so that for any $j$ we have 
\begin{equation}
\label{DUpper}
 \sup_{x,y}p_j(x,y)\leq 1/\ve_0,
 \end{equation}
  and the transition probabilities of the second step
   transition operators $R_j\circ R_{j+1}$ of $X_{j+2}$ given $X_j$ are bounded from below by $\ve_0$ (this is the uniform ellipticity condition): 
\begin{equation}
\label{DLower}
\inf_{j\geq1}\inf_{x,z}\int p_j(x,y)p_{j+1}(y,z)d\fm_{j+1}(y)\geq \ve_0.
\end{equation}
Next, for a uniformly bounded sequence of measurable functions $f_j:\cX_j\times\cX_{j+1}\to\bbR$ 
we set $Y_j=f_j(X_j,X_{j+1})$ and
\begin{equation}
\label{AddFunct}
S_n=\sum_{j=1}^n(Y_j-\bbE(Y_j)).
\end{equation}
Set $V_n=\text{Var}(S_n)$ and $\sig_n=\sqrt{V_n}$. Then by \cite[Theorem 2.2]{DS} we have 
$\DS \lim_{n\to\infty}V_n=\infty$ if and only if one can not decompose $Y_j$ as
$$ Y_j=\bbE(Y_j)+a_{j+1}(X_{j+1},X_{j+2})-a_{j}(X_j, X_{j+1})+g_n(X_j, X_{j+1})$$ 
 where  $a_n$ are uniformly bounded functions and $\DS \sum_j g_j(X_j, X_{j+1})$ converges almost surely.
 Note that when the chain is one step elliptic (i.e. $\ve_0\leq p_j(x,y)\leq \ve_0^{-1}$) and $f_j(x,y)=f_j(x)$ depends only on the first variable then by \cite[Proposition 13]{PelPTRF} we have 
$$
C_1\sum_{j=1}^{n}\text{Var}(f_j(X_j))\leq \sig_n^2\leq C_2\sum_{j=1}^{n}\text{Var}(f_j(X_j))
$$  
for some constants $C_1,C_2>0$, and so $\sig_n\to\infty$ iff the series $\sum_{j=1}^{\infty}\text{Var}(f_j(X_j))$ diverges.

 \begin{lemma}\label{LL}
 There exist a monotone increasing sequence $a_n$ and a bounded sequence $b_n$ so that 
 $\sig_n^2=a_n+b_n$ (and so $\sig_n=\sqrt{a_n}+O(\sig_n^{-1})$). 
 \end{lemma}
 
 \begin{proof}
 First, by \cite[Lemma 27]{DolgHaf} there are real  numbers $u_j$ so that $\sup_j|u_j|<\infty$ and
 \begin{equation}\label{1.1}
C:=\sup_{n<k}\left|\text{Var}(S_n-S_j)-\sum_{k=j+1}^n u_k\right|<\infty
 \end{equation}
 (in the notations of \cite{DolgHaf}, $u_j=\Pi_j''(0)$). 
 Next, as in \cite[Section 5]{DolgHaf}, for every $A$ large enough, there is a sequence of intervals $I_j=\{a_j,a_j+1,...,b_j\}$ in the integers whose union cover $\bbN$ so that $I_{j}$ is to the left of $I_{j+1}$ and the variance of $S_{I_j}=\sum_{k\in I_j}Y_k$ is between $A$ to $2A$. Moreover, if we set $k_n=\max\{k: b_k\leq n\}$ then $k_n\asymp \sig_n^2$ and
$$
\sup_{n}\left\|S_n-\sum_{j=1}^{k_n}S_{I_j}\right\|_{L^2}<\infty.
$$
Now, let us set $U_j=\sum_{k\in I_j}u_j$. Then by \eqref{1.1}, 
$$
\left|U_j-\text{Var}(S_{I_j})\right|\leq C.
$$
Note that  because of the properties of the blocks $S(I_j)$ and the exponential decay of correlation (see \cite[Proposition 1.11 (2)]{DS}) we have that 
\begin{equation}\label{2}
\sig_n^2=\sig_{k_n}^2+O(1).
\end{equation}
Next, let us set $V_k=\text{Var}\left(\sum_{j=1}^{k}S_{I_j}\right)$. Then $\sig_{k_n}^2=V_{k_n}$. Moreover, by \eqref{1.1} we have
$$
\sup_{k}\left|V_{k+1}-\sum_{j=1}^{k+1}U_k\right|\leq C<\infty
$$
and so 
$$
V_{k+1}-V_k\geq U_{k+1}-2C\geq A-2C.
$$
Thus when $A>2C$ we see that $V_k$ is increasing. Since
$$
\sig_n^2=\sig_{k_n}^2+O(1)=V_{k_n}+O(1)
$$
and $k_n$ is increasing we see that we can take $a_n=\sig_{k_n}^2=V_{k_n}$.
 \end{proof}

Next, in \cite[Proposition 24]{DolgHaf} we have shown that Assumption \ref{GrowAssum} holds true for all $m$. By applying Lemma \ref{LL} together with Theorem \ref{BE1} and Corollary \ref{ASIP}  we get the following result.

\begin{theorem}
Theorems   \ref{BE1} holds true for every $m$.  Furthermore, for each $n$ and $m$ there is a coupling of $(Y_1,...,Y_n)$ and a zero mean iid (finite) Gaussian sequence $(Z_1,...,Z_n)$ such that $\text{Var}(Z_j)=b_j$ and 
$$
\left\|S_n-\sum_{j=1}^n Z_j\right\|_{m-1}\leq C
$$
where $C$ is a constant which does not depend on $n$.
\end{theorem}

\begin{remark}
When the variance of $S_n$ grows linearly fast then Assumption \ref{GrowAssum} holds true for every $m$ for any sufficiently fast mixing Markov chain (see \cite{SaulStat}), and thus the above theorem holds. However, in this section we are interested in the situation when $\sig_n^2$ can grow arbitrary slow.
\end{remark}

\subsection{Edgeworth expansions}
We prove here the following:

\begin{theorem}\label{High thm MC}
(i) Suppose that $\|f_n\|_{L^\infty}=O(n^{-\be})$ for some $\be\in(0,1/2)$. Then  Assumption \ref{DerAss} holds true when $r:=m-2<\frac1{1-2\be}$, and therefore the  non-uniform Edgeworth expansions of order $r$ and the  Edgeworth expansions in transport distances of order $r$ (as in Theorems \ref{Edge2} and \ref{Edge1}) and the expansions of the expectations as in Corollary \ref{Cor} (ii) hold true (with such $m$). In particular, if $\|f_n\|_{L^\infty}=O(n^{-1/2})$ then $S_n$ obeys expansions of all orders (in all of the above senses).

Moreover, the condition $r=m-2<\frac1{1-2\be}$ is optimal even for independent summands.
\vskip0.12cm
(ii) Suppose that $\cX$ is a compact Riemannian manifold, and that $f_n$ are uniformly H\"older continuous with some exponent $\al$. Additionally, assume that all the measures $\mathfrak{m_j}$ coincide with the volume measure and the transition densities $p_j$ which are (uniformly) bounded and bounded away from the origin. Then the non-uniform Edgeworth expansions of order $r=m-2$, 
the Edgeworth expansions in transport distances of order $r$, and the  expansions of the expectations as in Corollary \ref{Cor}(ii) hold true for every $m$ so that $r=m-2<\frac{1+\al}{1-\al}$, where for $\al=1$ we set $\frac{1+\al}{1-\al}=\infty$ (in particular, for Lipschitz continuous functions we get expansions of all orders in all of the above senses).

Moreover, the condition $r=m-2<\frac{1+\al}{1-\al}$ is optimal even in the iid case.
\end{theorem}

\begin{proof}
Our goal is to verify condition \eqref{psi cond}  in both parts of Theorem \ref{High thm MC}. 
The beginning of the proof of both parts is identical.

First, similarly to \cite[Section 5]{DolgHaf}, there are intervals  $I_{j,n}$, $j\leq k_n$  in the positive integers whose union cover $\{1,2,...,n\}$ and $k_n\asymp \sig_n^2$.
Moreover, for every $p\geq 1$ the $L^p$-norms of $S_{I_{j,n}}$ are bounded by some constant $C_p$ which does not depend on $n$ and $j$. Now, the $k$-th derivative of the characteristic function has the form
\begin{equation}\label{Begg1}
\frac{d^m}{dt^m}\,\bbE[e^{itS_n}]=i^{m}\bbE[S_n^m e^{itS_n}]=i^{m}\bbE\left[\left(\sum_{j=1}^{k_n}S_{I_{j,n}}\right)^{m}e^{itS_n}\right]=i^{m}\sum_{1\leq \ell_1,...,\ell_{m}\leq k_n}\bbE[X(\bar\ell)e^{itS_n}]
\end{equation}
where for a given $\bar\ell=(\ell_1,...,\ell_{m})$, if $\ell_{s_q}, q=1,2,...,m(\bar\ell)$ are the nonzero ones among $\ell_s$ with $s_j<s_{j+1}$ and $m(\bar\ell)\leq m$ then with $\Xi_j=\Xi_{j,n}=S_{I_{j,n}}$,
$$
\Xi(\bar\ell)=\prod_{q=1}^{m(\ell)}\Xi_{s_q}^{\ell_{s_q}}.
$$
Let $J_{i}$ denote the gap between $I_{s_{i-1},n}$ and $I_{s_i,n}$, where we set $I_{0,n}=\{0\}$. With $\bar m=m(\bar\ell)$,  let $J_{\bar m+1}$ be the complement of the union of $J_1,...,J_{\bar m}$ and $I_{s_j,n}$. For $I=\{a,a+1,...,a+b\}$ let
$$
R_{it}^{I}=R_{a,it}\circ R_{a+1,it}\circ\cdot\circ R_{a+n,it}
$$
where the operator $R_{j,it}$ is given by $R_{it,j}(x)=\bbE[e^{itf_j(X_j,X_{j+1})}g(X_{j+1})|X_j=x]$.
Using the Markov property we see that, for a fixed $\bar\ell$, with $m=m(\bar\ell)$ we have 
$$
\bbE[\Xi(\bar\ell)e^{itS_n}]=\bbE[R_{it}^{J_1}\circ L_{s_1}\circ R_{it}^{J_2}\circ L_{s_2}\circ\cdots \circ L_{s_{\bar m-1}}\circ R_{it}^{J_{\bar m}}\circ L_{s_{\bar m}}\circ R_{it}^{J_{\bar m+1}}\textbf{1}]
$$
where $L_{s}(g)=\bbE[g(\Xi_s)e^{it \Xi_{s}}\Xi_{s}^{\ell_{s}}|X_{a_s-1}]$ and $a_s$ is the left end point of $I_s$. 
\vskip0.1cm 

We claim next that  the operator norms of the operators $L_{s_j}$ with respect the the appropriate essential supremum norms  
are bounded by some constant which does not depend on $\bar\ell$ or $n$. 

We first need the following.
Let $(\Om,\cF,\bbP)$ be a probability space.
Recall that the $\psi$-mixing coefficient of two sub $\sig$-algebras $\cG,\cH$ of $\cF$  is given by 
$$
\psi(\cG,\cH)=\sup\left\{\left|\frac{\bbP(A\cap B)}{\bbP(A)\bbP(B)}-1\right|: A\in\cG, B\in\cH, \bbP(A)\bbP(B)>0\right\}.
$$
Then the $\psi$-mixing sequence $\psi(n)$ of the Markov chain $\{X_j\}$  is given by 
$$
\psi(n)=\sup\{\psi(\cF_{k},\cF_{k+n,\infty}): k\in\bbN\}
$$
where $\cF_k$ is the $\sig$-algebra generated by $\{X_1,...,X_k\}$ and $\cF_{k+n,\infty}$ is the $\sig$-algebra generated by $\{X_{j}: j\geq k+n\}$. Then by \cite[Proposition 1.22]{DS},  $\psi(n)\leq Ce^{-cn}$ for some $c,C>0$, and in particular $\psi(1)<\infty$. Now by \cite[Ch. 4]{Br}, 
$$
\psi(\cG,\cH)=\sup\left\{\|\bbE[h|\cG]-\bbE[h]\|_{L^\infty}: h\in L^1(\Om,\cH,\bbP),\,\|h\|_{L^1}\leq1\right\}.
$$
Using the above representation of $\psi(\cdot,\cdot)$ and the definition of $\psi(1)$
we see that 
$$
\left\|L_{s_j}(g)-\bbE[g(\Xi_{s_j})e^{it \Xi_{s_j}}\Xi_{s_j}^{\ell_{s_j}}]\right\|_{L^\infty}\leq \psi(1)\|g(\Xi_{s_j})e^{it \Xi_{s_j}}\Xi_{s_j}^{\ell_{s_j}}\|_{L^1}\leq\psi(1)\|\Xi_{s_j}\|_{L^{\ell_{s_j}}}^{\ell_{s_j}}\|g(\Xi_{s_j})\|_{L^\infty} $$$$\leq C'_m\|g(\Xi_{s_j})\|_{L^\infty}
$$
for some constant $C'_m$, where we have used that $\ell_{s_j}\leq m$ and all the  $L^p$ norms of $\Xi_j$ are uniformly bounded in $j$.
To complete the proof of the claim, note that 
$$
\left|\bbE[g(\Xi_{s_j})e^{it \Xi_{s_j}}\Xi_{s_j}^{\ell_{s_j}}]\right|\leq \|g(\Xi_{s_j})\|_{L^\infty}\|\Xi_{s_j}\|_{L^{\ell_{s_j}}}^{\ell_{s_j}}\leq C_m''\|g(\Xi_{s_j})\|_{L^\infty}.
$$
\vskip0.1cm

Using the above claim, we conclude that
$$
\left|\bbE[\Xi(\bar\ell)e^{itS_n}]\right|\leq C''_m\prod_{s=1}^{m(\bar\ell)+1}\|R_{it}^{J_s}\|_{L^\infty}.
$$
Let us now bound the norms $\|R_{it}^{J_s}\|_{L^\infty}$. First we note that these norms are always smaller than one. Second, since $m(\bar\ell)\leq m$,  the $L^2$-norm of the sum of  $S_{I_{s_j}}, j\leq m(\bar\ell)$ is bounded by some constant which does not depend on $\bar\ell$ or $n$. Thus, the $L^2$ norm along the sum of the  completing blocks $J_s$ is of order $\sig_n$. Hence at one least of the these norms is of order $\sig_n$ (there are at most $m+1$ blocks $J_{s_j}$). Let us denote this block by $J_u$.

In the circumstances of part (i), in the beginning of \cite[Section 6]{DolgHaf} we showed\footnote{The upper bounds on the characteristic function obtained there relied on appropriate estimates on the norms of the operators.} that if $r<\frac{1}{1-2\be}$ then when 
$$
|t|\leq C(\sig(u))^{r-1}, \,\sig(u):=
\left\|\sum_{I_{j,n}\subset J_u}S_{I_{j}}\right\|_{L^2}
$$ 
(for some constant $C$) we have 
$$
\|R_{it}^{J_u}\|_{L^\infty}\leq e^{-ct^2\sig^2(u)}
$$
where $c>0$ is another constant. Note that in \cite{DolgHaf} we considered the supremum norms, 
but we can always replace the supemum norms on $\cX_j$ with the essential supremum norm with respect to the law of $X_j$.
Recall now that, $\sig(u)\geq c'\sig_n$ for some positive constant $c'$, and we thus conclude that 
$$
\left|\bbE[\Xi(\bar\ell)e^{itS_n}]\right|\leq C''_me^{-c''t^2\sig_n^2},\,\,c''>0.
$$
Starting from \eqref{Begg1} and then
combining this with the previous estimates we see that when $|t|\leq C''\sig_n^{r-1}$ we have 
$$
\left|\frac{d^m}{dt^m}\,\bbE[e^{itS_n}]\right|\leq A'_m\sig_n^{2m}e^{-c''t^2\sig_n^2}
$$
for some constant $A'_m$,
which is  enough to verify  \eqref{psi cond}. Thus, as noted in Remark \ref{Rem psi}, Assumption \ref{DerAss} is satisfied   when $r=m-2<\frac{1}{1-2\be}$.

Next, in the circumstances of part (ii),  in the proof of \cite[Proposition 34]{DolgHaf} we showed  that when $|t|$ is large enough then for every $n$,
$$
\|R_{it}^{J_u}\|_{L^\infty}\leq Ce^{-c|t|^{1-1/\al}\sig^2(u)}\leq Ce^{-c'|t|^{1-1/\al}\sig_n^2}
$$
which, starting again from \eqref{Begg1}, yields that 
$$
\left|\frac{d^m}{dt^m}\,\bbE[e^{itS_n}]\right|\leq R_m\sig_n^{2m}Ce^{-c'|t|^{1-1/\al}\sig_n^2}
$$
for some constant $R_m$,
which is also enough for Assumption \ref{DerAss} to hold true  when $r=m-2<\frac{1+\al}{1-\al}$ (using again \eqref{psi cond}).

Finally, the optimality of the conditions was already discussed in the context of uniform Edgeworth expansions in \cite{DolgHaf}. 

\end{proof}

\subsection{Random (partially expanding or hyperbolic) dynamical systems}\label{RDS}

Let $(\Om,\cF,\bbP)$ be a probability space and let $\theta:\Om\to\Om$ be an ergodic invertible map.
Let $(\cX,d)$ be a metric space and let $\cE_\om, \om\in\Om$ be a measureable family of compact subsets (see \cite[Ch. 6]{HK}). Set $\cE=\{(\om,x): \om\in\Om, x\in\cE_\om\}$ and let $T:\cE\to\cE$ be a measurable map of the form
$$
T(\om,x)=(\theta\om, T_\om x)
$$
where
$T_{\om}:\cE_\om\to\cE_{\theta\om}$ is a random family of maps ($T$ is the so-called skew product induced by $\{T_\om\}$). Let $\mu$ be a $T$-invariant probability measure on $\cE$, and let us represent it in the form 
$$
\mu=\int \mu_\om d\bbP(\om)
$$ 
where for $\bbP$-almost every $\om$ the probability measure $\mu_\om$ on $\cE_\om$ satisfies $(T_\om)_*\mu_\om=\mu_{\te\om}$. 

Next, let $\cL_\om$ denote the dual of $T_\om$ with respect to the measures $\mu_\om$ and $\mu_{\theta\om}$, namely the unique operator so that 
$$
\int g\cdot(f\circ T_\om)d\mu_\om=\int f\cdot(\cL_\om g)d\mu_{\te\om}
$$
for all bounded functions $g$ and $f$ on the appropriate domains.
Let $(B_\om,\|\cdot\|_\om)$ be a norm on functions on $\cE_\om$ and let $f:\cE\to\bbR$ be a measurable function so that $\|f\|:=\text{ess-sup}\|f(\om,\cdot)\|_\om<\infty$. In what follows we will introduce assumptions on $\cL_\om$ and their complex perturbations which will guarantee that Assumption \ref{GrowAssum} holds true for 
$$
S_n^\om=\sum_{j=0}^{n-1}f_{\te^j\om}\circ T_{\te^{j-1}\om}\circ\cdots\circ T_{\om},\,\,\,f_\om(\cdot)=f(\om,\cdot)
$$
for $\bbP$-a.e. $\om$,  when considered as a random variables on the space $(\cE_\om,\mu_\om)$.

For a complex parameter $z$, let us consider the operator $\cL_\om^{z}$ given by $\cL_\om^{z}(g)=\cL_{\om}(ge^{zf_\om})$. For it to be well defined, we  assume that $\cL_\om$ is a continuous operator between $B_\om$ and $B_{\te\om}$ and that the map $z\to e^{zf_\om}\in\cB_\om$ is analytic in $z$, uniformly in 
$\om$.

\begin{assumption}\label{AssRDF}
There is a constant $r_0>0$ so that ($\bbP$-a.s.) for every complex $z$ with $|z|<r_0$ there is a triplet $\la_\om(z)\in\bbC\setminus\{0\}$, $h_\om^{(z)}\in B_\om$ and $\nu_\om^{(z)}\in B_\om^*$ which is measurable in $\om$, analytic in $z$ and: 
\begin{enumerate}
\item $\la_\om(0)=1$, $h_\om^{(0)}=\textbf{1}$, $\nu_\om^{(0)}=\mu_\om$, $\nu_\om^{(z)}(h_\om^{(z)})=1$;
\vskip0.1cm
\item with 
$$\cL_{\om}^{z,n}=\cL_{\te^{n-1}\om}^{z}\circ\cdots\circ\cL_{\om}^{z}$$
and 
$$
\la_{\om,n}(z)=\la_{\te^{n-1}\om}(z)\cdots\la_{\om}(z)
$$
there are $C>0$ and $\del\in(0,1)$ so that
$$
\left\|(\la_{\om,n}(z))^{-1}\cL_{\om}^{z,n}-\nu_{\om}^{(z)}\otimes h_{\te^n\om}^{(z)}\right\|_{\te^n\om}\leq C\del^n
$$
where the operator $\nu\otimes h$ is given by $g\to \nu(g)\cdot h$.
\end{enumerate}
\end{assumption}

\begin{example}
Let us list a few types of random maps $T_\om$ and measures $\mu_\om$ for which Assumption \ref{AssRDF} holds true.
\begin{enumerate}
\item The random expanding maps  considered in \cite[Ch. 6]{HK}, where $\|\cdot\|_\om$ is a H\"older norm of a function on $\cE_\om$ with respect to some exponent $\al\in(0,1]$ and $\mu_\om$ is a random Gibbs measure;
\vskip0.1cm
\item The random expanding maps  considered in \cite{DavorCMP}, where $\|\cdot\|_\om$ is the bounded-variation norm and $\mu_\om$ is the unique random  absolutely continuous equivariant measure.
\vskip0.1cm
\item The hyperbolic maps considered in \cite{DavorTAMS} (see also \cite{DDH}), with $\|\cdot\|_\om$ being the appropriate strong norm and $\mu_\om$ being the unique random fiberwise absolutely continuous equivariant measure.
\vskip0.1cm
\item The random partially expanding or hyperbolic maps considered in \cite{HafYT}, with $\|\cdot\|_\om$ being the appropriate ``weighted" H\"older norm and $\mu_\om$ being a sampling measure. 
\vskip0.1cm
\item The uniformly random version (as in \cite{RPF2022}) of the random partially expanding maps considered in \cite{Varandas}  with $\|\cdot\|_\om$ being a H\"older norm with respect to some exponent $\al\in(0,1]$ and $\mu_\om$ being a random Gibbs measure corresponding to a potential with a sufficiently small oscillation (see \cite{RPF2022}). 
\end{enumerate}
\end{example}

\begin{remark}
We note that in \cite{DavorCMP, DavorTAMS} Assumption \ref{AssRDF} appears as 
$$
\left\|\cL_{\om}^{z,n}-\la_{\om,n}(z)\left(\nu_{\om}^{(z)}\otimes h_{\te^n\om}^{(z)}\right)\right\|_{\te^n\om}\leq C\del^n
$$
but since $\la_\om(z)=1+O(z)$ upon decreasing $r_0$ we get that  $(1-\ve)^n\leq |\la_{\om,n}(z)|\leq (1+\ve)^n$ for an arbitrary small $\ve$, and hence the above two forms are equivalent. 
\end{remark}

Next, in \cite[Proposition 34]{NonU BE} we proved the following result.
\begin{proposition}\label{FirstProp}
Under assumption \ref{AssRDF} we we have tha following.

(i) There is a nonnegative number $\sig$ so that $\bbP$-a.s. 
$$
\sig^2=\lim_{n\to\infty}n^{-1}\text{Var}_{\mu_\om}(S_n^\om).
$$
 Moreover, $\sig^2=0$ iff $f(\om,x)-\int f(\om,y)d\mu_\om(y)=r(\om,x)-r(\te\om, T_\om x)$ for some function $r$ so that $\int|r(\om,y)|^2d\mu_\om(y)d\bbP(\om)<\infty$.

(ii) Suppose that $\sig^2>0$. Then, under  Assumption \ref{AssRDF}, for $\bbP$-a.e. $\om$ we have that
 $S_n=S_n^\om-\mu_\om(S_n^\om)$ verifies Assumption \ref{GrowAssum} with every $m$.  Thus, all the results stated in Theorem \ref{BE1} and Corollary \ref{Cor}  hold true for all $m$.
\end{proposition}

\subsubsection{Coupling: on the verification of the conditions of Corollary\ref{ASIP}}

\begin{lemma}
Suppose that $\sig>0$.
For $\bbP$-a.a. $\om$ there exists an increasing sequence $(B_n(\om))_{n=1}^\infty$ so that 
$$
\sig_{\om,n}^2=\|S_n^\om-\mu_\om(S_n^\om)\|_{L^2(\mu_\om)}^2=B_n(\om)+O(1).
$$
As a consequence, the couplings with the properties described in Corollary \ref{ASIP} always exist (for all $m$) for  $S_n=S_n^\om-\mu_\om(S_n^\om)$.
\end{lemma}
\begin{proof}
First, like in the stationary case, we have 
\begin{equation}\label{Ab}
\ln \bbE[e^{zS_n^\om}]=\Pi_{\om,n}(z)+H_{\om,n}(z) +O(|z|\del^n).
\end{equation}
By differentiating twice both sides of \eqref{Ab} and using the Cauchy integral formula we see that 
$\sig_{\om,n}^2=\|S_n^\om-\mu_\om(S_n^\om)\|_{L^2(\mu_\om)}^2$ has the form 
$$
\sig_{\om,n}^2=\sum_{j=0}^{n-1}u_{\te^j\om}+O(1), \,\,u_\om=\Pi_\om''(0)
$$
with $\text{ess-sup}|u_\om|<\infty$. Moreover, $\sig_{\om,n}^2\asymp \sig^2 n$ for some $\sig>0$. Thus, for $\bbP$-a.a. $\om$ there are intervals in the integers $I_{\om,1}, I_{\om,2},...$ whose union covers $\bbN_+=\{0,1,2,...\}$ and $I_{\om,j}$ is to the left of $I_{\om,j+1}$ and with $U_{\om,k}=\sum_{j\in I_{\om,k}}u_{\te^j\om}$ we have $A\leq U_{\om,k}\leq 2A$ for some constant $A>1$. Moreover, for every 
interval $J$ in the integers  with the same left end point as $I_{j,\om}$ so that $J\subset I_{j,\om}$ and $J\not= I_{j,\om}$ we have 
$\sum_{s\in I_{j,\om}}u_{\te^s\om}\leq A$.

Let $k_n(\om)=\max\{k: I_{\om,k}\subset\{0,1,...,n-1\}$. Then $(k_n(\om))_{n=1}^{\infty}$ is an increasing sequence and 
$$
\sig_{\om,n}^2=\sum_{j=1}^{k_n(\om)}U_{\om,j}+O(1)
$$
where the $O(1)$ term is actually bounded in absolute value by $A$.
Now we can take $B_n=B_n(\om)=\sum_{j=1}^{k_n(\om)}U_{\om,j}$.
\end{proof}
\subsubsection{On the verification of Assumption \ref{DerAss} with $m=3$}
We first need the following assumption.
\begin{assumption}\label{Per}
For every compact subset $K$ of $\bbR\setminus\{0\}$ we have 
\begin{equation}\label{Per1}
\sup_{t\in K}\sup_{n\geq 1}\|\cL_\om^{it,n}\|_{\te^n\om}\leq C(K)
\end{equation}
where $C(K)$ depends only on $K$
and for some $\alpha>0$
\begin{equation}\label{Per2}
\sup_{t\in K}\|\cL_\om^{it,n}\|_{\te^n\om}\leq C_\om(K)e^{-c_K n^\alpha}
\end{equation}
where $C_\om(K)>0$ and $c_K>0$ depend on $K$ and $\om\to C_\om(K)$ is measurable.
\end{assumption}

\begin{proposition}\label{Prop49}
Under the additional Assumption \ref{Per}, for $\bbP$-a.e. $\om$  for every $m$ and all $b>a>0$ we have 
$$
\int_{a\sig_{\om,n}\leq |t|\leq b\sig_{\om,n}}\left|\frac{f_n^{(m)}(t)}{t}\right|dt=O(\del^n)
$$
where $f_n(t)=\bbE_{\mu_\om}[e^{it (S_n^\om-\bbE_{\mu_\om}[S_n^\om])}]$ and $\del\in(0,1)$ depends on $a,b$ and $\om$. Thus, under Assumptions \ref{AssRDF} and \ref{Per} all the results stated in Theorems \ref{Edge1}, \ref{Edge2} and Corollary \ref{Cor} hold true with $m=3$.
\end{proposition}

\begin{proof}
Let us take $K=[-b,-a]\cup[a,b]$. Let $A\subset \Om$ be a measurable set with positive probability so that on $A$ we have $C_K(\om)\leq C_0$ for some constant $C_0$.
 Since $\te$ is ergodic  it follows from Kac's formula that for $\bbP$-a.e. $\om$ there is an infinite sequence of natural numbers $n_1(\om)<n_2(\om)<...$ so that $\te^{n_i(\om)}\om\in A$ and 
\begin{equation}\label{n k}
\lim_{k\to\infty}\frac{n_k(\om)}{k}=\frac{1}{\bbP(A)}.
\end{equation}

Now, after replacing $f(\om,x)$ with $f(\om,x)-\int f(\om,y)d\mu_\om(y)$ (i.e. assuming $\mu_\om(f_\om)=0$)
we have 
$$
\frac{d^k}{dt^k}\bbE[e^{it S_n^\om}]=i^k\bbE[(S_n^\om)^ke^{itS_n^\om}]=
i^k\sum_{1\leq \ell_1\leq \ell_2\leq...\leq \ell_k\leq n}\bbE[X_{\bar \ell,\om}(\om)e^{it S_n^\om}]
$$
where
$$
X_{\bar\ell,\om}=\prod_{s=1}^{k}f_{\te^{\ell_s}\om}\circ T_{\om}^{\ell_s}.
$$
Let us fix some $\bar\ell=(\ell_1,...,\ell_k)$.
Let  $j_1<j_2<...<j_r$, $r\leq k$ be the distinct indexes among $\ell_1,...,\ell_k$, and suppose that $j_s$ appears $a_s$ times in $\bar\ell$.
Let us define an operator by $L_{\om,j_s,it}$ by
$$
L_{\om,j_s,it}(g)=\cL_{\te^{j_s}\om}(g f^{a_s}_{\te^{j_s}\om}e^{it f_{\te^{j_s}\om}}).
$$
Then the norms of the operators $L_{\om,j_s,it}$ are uniformly bounded by some constant $R(K)\geq1$ when $t\in K$.
Notice now that
$$
\bbE[X_{\bar \ell,\om}e^{it S_n^\om}]=\mu_{\te^n\om}(\cL_{\te^{j_r+1}\om}^{it,n-j_r}\circ L_{\om,j_r,it}\circ\cdots\circ \cL_{\te^{j_1+1}\om}^{it,j_2-j_1-1}\circ L_{\om,j_1,it}\circ \cL_{\om}^{it,j_1}\textbf{1}).
$$
Now, since $\sum_{j=1}^r (j_{s+1}-j_{s}-1)=n-r$, where $j_{r+1}=n$ and $j_0=-1$, at least one of the the iterates  
$\cL_{\te^{j_s+1}\om}^{it,j_{s+1}-j_s-1}$ is a composition of at least $[n/r]-1$ operators. Now, by \eqref{n k} for every interval in the integers $I\subset\{1,2,...,n\}$, whose left end point is large enough and whose length is at least $[n/3r]-3$, contains at least one point $a$ so that $\te^a\om\in A$. By taking  $I$ to be the  middle third of the interval $I=\{j_s+1,...,j_{s+1}\}$ whose length is at least $[n/r]-1$, using \eqref{Per2} with $\te^a\om$ instead of $\om$ and $[n/3r]-2$ instead of $n$ 
and using \eqref{Per1} to bound the norms of the other ``big" blocks we see that $\bbP$-a.s. for every $n$ large enough we have
$$
\left\|\cL_{\te^{j_r+1}\om}^{it,n-j_r}\circ L_{\om,j_r,it}\circ\cdots \circ L_{\om,j_1,it}\circ \cL_{\om}^{it,j_1-1}\right\|_{\te^n\om}\leq (R(K))^{k}(1+C)^{k}e^{-\frac{1}{4k}c_Kn}
$$
for some $C>0$, where we have used that $r\leq k$.  Hence, for all $k$ and positive $b>a$ we have 
$$
\int_{a\leq |t|\leq b}\left|\frac{\psi_{\om,n}^{(k)}(t)}{t}\right|dt=O(\del^{n^\alpha})
$$
for some $\del\in(0,1)$, where $\psi_{\om,n}(t)=\bbE[e^{it(S_n^\om-\mu_\om(S_n^\om))}]$. As noted in Remark \ref{Rem psi}, this is enough for Assumption \ref{DerAss} to hold true with $m=3$.

\end{proof}

\begin{example}
We refer to \cite[Ch.6]{HK} and \cite{DolgHaf LLT} for sufficient conditions for Assumption \ref{Per}.
\end{example}

\subsubsection{On the verification of Assumption \ref{DerAss} with $m>3$}
In order to control the integral of $|f_n^{(m)}(t)|/|t|$ over domains of the form $\{|t|\in[c\sig_n, B\sig_n^{m-2}]\}$ for large $c$'s we need to introduce a few additional assumptions.
Let $\|\cdot\|_{1,\om}$ be (possibly) another norm on the space of functions on $\cE_\om$ so that $\|\cdot\|_{1,\om}\geq \|\cdot\|_{\infty}$. For every fixed $t\not=0$ let us take a norm $\|\cdot\|_{\om,(t)}$ so that 
\begin{equation}\label{b rel}
c_1\|g\|_{\om,(t)}\leq \|g\|_{1,\om}\leq c(1+|t|)\|g\|_{\om,(t)}
\end{equation}
and that 
\begin{equation}\label{b Bound}
\|\cL_{\om}^{it,n}\|_{\te^n\om, (t)}\leq C
\end{equation}
for some constant $C$ which does not depend on $\om,n$ and $t$. 
\begin{proposition}\label{Der Lem RDS}
Let  Assumptions \ref{AssRDF}  and \ref{Per} hold, and suppose that $\sig>0$.
Assume also that  there are positive random variables $\rho_\om,\gamma(\om)$ and $b_\om$ so that, $\bbP$-a.s. for every $t\in\bbR$ such that $|t|\geq b_\om$ we have 
\begin{equation}\label{DINQ}
\|\cL_{\om}^{it,n_\om(t)}\|_{\te^n\om,(t)}\leq e^{-n_\om(t)\gamma(\om)}
\end{equation}
where $n_\om(t)=[\rho_\om \ln |t|]$. Then  for $\bbP$-a.e. $\om$ we have that $S_n^\om$ satisfies the integrability conditions in Assumption \ref{DerAss}  for every $m$.  Thus, under Assumptions \ref{AssRDF}, \ref{Per} and the above condition all the Edgeworth expansions stated in Theorems \ref{Edge2}, \ref{Edge1} and Corollary \ref{Cor} (ii) hold true for every $m$.
\end{proposition}
\begin{proof}
First, by Proposition \ref{Per}, it is enough to show that the integral of $|f_n^{(m)}(t)|/|t|$ over $\{|t|\in [b_1\sig_n, B\sig_n^{m-2}]\}$ is $o(\sig_n^2)$  for all $b_1>0$ large enough and $B>0$. Next, by Remark \ref{Rem psi}, it is enough to verify \eqref{psi cond} with $c$ large enough.

Let $A$ be a measurable subset of $\Om$ with positive probability so that $b_\om\leq b_0$ and $\rho_1\leq \rho_\om\leq \rho_2$ and $\gamma(\om)\geq \gamma$ for all $\om\in A$, where $b_0, \rho_1, \rho_2$ and $\gamma$ are positive constants. Since $\te$ is ergodic we know  from Kac's formula that for $\bbP$-a.e. $\om$ there is an infinite sequence of natural numbers $n_1(\om)<n_2(\om)<...$ so that $\te^{n_i(\om)}\om\in A$ and 
$$
\lim_{k\to\infty}\frac{n_k(\om)}{k}=\frac{1}{\bbP(A)}.
$$
Thus,
with
$$
m_n(\om)=\max\{m: n_m(\om)\leq n\}
$$
we have
$$
\lim_{n\to\infty}\frac{m_n(\om)}{n}=\bbP(A).
$$

Now, as in the proof of Proposition \ref{Prop49} we have 
\begin{equation}\label{EE1}
\frac{d^k}{dt^k}\bbE[e^{it S_n^\om}]=i^k\bbE[(S_n^\om)^ke^{itS_n^\om}]=
i^k\sum_{1\leq \ell_1\leq \ell_2\leq...\leq \ell_k\leq n}\bbE[X_{\bar \ell,\om}(\om)e^{it S_n^\om}].
\end{equation}
where
$$
X_{\bar\ell,\om}=\prod_{s=1}^{k}f_{\te^{\ell_s}\om}\circ T_{\om}^{\ell_s}.
$$
As in the proof of Proposition \ref{Prop49}, let us fix some $\bar\ell=(\ell_1,...,\ell_k)$ and
 assume that $j_1<j_2<...<j_r$ are the distinct indexes among $\ell_1,...,\ell_k$. Let $a_s$ be the number of times that $j_s$ appears  in $\bar\ell$.
Then we have
\begin{equation}\label{EE2}
\bbE[X_{\bar \ell,\om}(\om)e^{it S_n^\om}]=\mu_{\te^n\om}(\cL_{\te^{j_r+1}\om}^{it,n-j_r}\circ L_{\om,j_r,it}\circ\cdots\circ \cL_{\te^{j_1+1}\om}^{it,j_2-j_1-1}\circ L_{\om,j_1,it}\circ \cL_{\om}^{it,j_1}\textbf{1}).
\end{equation}
Let us suppose now that  $|t|\leq c_0 n^{(m-3)/2}$ for some $c_0>0$.
In order to bound the norm of the above composition, let us note that there is an $\ve_0>0$ so that for every $n$ large enough there are at least $\ve_0n$ indexes $k$ between $0$ and $n-1$ so that $\te^k\om\in A$. Let us take some $D>0$, and let $k_1,k_2,...,k_{d_n}$, $d_n\asymp \frac{n}{\ln n}$ be indexes so that $\te^{k_i}\om\in A$ and $k_{j+1}-k_j\geq D\ln n$. By omitting the $k_j$'s which belong to $\cJ_s=[j_s-D\ln n, j_s]$ for some $s$, we can always assume that all $k_j$'s are not in the union of $\cJ_s$. Now, if $D$ is large enough then by using that 
$$
\left\|\cL_{\te^{k_j}\om}^{it, [\rho_{\te^{k_j}\om}\ln |t|]}\right\|_{\te^{k_j}\om,(t)}\leq e^{-\gamma[\rho_{\te^{k_j}\om}\ln |t|]}
$$
and  \eqref{b Bound} to bound the other big blocks, taking into account thta
 $\rho_1\leq \rho_{\te^{k_j}\om}\leq \rho_2$ we conclude that the $\|\cdot\|_{\te^n\om,(t)}$ norm of the product of the operators does not exceed 
$$
C_0R^{n/\ln n}e^{-c_1\ln |t|n/\ln n}
$$
where $C_0, R$ and $c_1$ are constants which do not depend neither $t$ nor $n$, and we have assumed that $|t|\geq b_0$. Now, using also \eqref{b rel}  we conclude that if $b_1$ is large enough and $b_1\leq |t|\leq c_0'\sig_{\om,n}^{m-3}$ then 
$$
\left|\frac{d^k}{dt^k}\bbE[e^{it S_n^\om}]\right|\leq C(1+|t|)n^k e^{-c_2n/\ln n}
$$
for some constant $c_2>0$. Hence  if $c\geq b_1$ then \eqref{psi cond} is valid for all $m$ and an arbitrary large $B$. 
\end{proof}

\begin{example}
\,

\begin{itemize}
\item The conditions of Proposition \ref{Der Lem RDS} hold true for the random expanding interval maps considered in \cite[Section 5.1.3]{Ha}, with $\|\cdot\|_{\om}$ being the $\text{BV}$ norm (note that the maps considered in \cite[Section 5.1.3]{Ha} can have discontinuities).
\vskip0.1cm
\item 
Let $T_1,T_2,...,T_d$ be smooth expanding maps on some compact Riemannian manifold with the same properties as in \cite[Section 3.4]{DolgIsr} and let $f_1,...,f_d$ be smooth functions so that $f_1$ is not infinitesimally integrable (see also \cite[Section 6.5]{FL}). Now let us suppose that $(\Om,\cF,\bbP,\te)$ is the shift system generated by some ergodic sequence of random variable $\xi_j$ taking values on $\{1,2,...,d\}$ so that $\bbP(\xi_1=\xi_2=...=\xi_{n_0}=1)>0$  for some  $n_0$ that will be described below,  and assume that for $\om=(\xi_j)$ we have $T_\om=T_{\xi_0}$ and $f_\om=f_{\xi_0}$ (e.g. we can take a tationary finite state Markov chain with positive transition probabilities).

Next, by \cite[Lemma 3.18]{DolgIsr} the, so-called Dolgopyat type inequality (\cite[(6.9)]{FL})  holds true for the transfer operators $\cL_{1,it}$ and the norms $\|\cdot\|=\|\cdot\|_{C^1}$ and 
$$
\|g\|_{(t)}=\max\left(\sup|g|, \frac{\sup|Dg|}{C(1+|t|)}\right),
$$
where the complex transfer $\cL_{j,z}$ operators are given by $g\to \cL_{T_j}(e^{z f_j}g)$ and $C$ is a sufficiently large constant.
Namely, $\cL_{1,it}$ obeys  a nonrandom version of \eqref{DINQ} with these norms. 
Now, note that in view of the Lasota-York inequality (see \cite[Chapter 6]{HK} for the random case) if $C$ is large enough then 
$$
\|\cL_{\om}^{it,n}\|_{(t)}\leq 1\,\text{ and }\,\|\cL_{j,it}^n\|_{(t)}\leq1.
$$ 
Now, as argued in \cite[Section 5.4]{FL}, there are constants $c_0,K_0>0$ ($K_0$ can be taken to be arbitrary small) and $r_0\in(0,1)$ so that if $|t|\geq K_0$ then
$$
\|\cL_{1,it}^n\|_{(|t|)}\leq c_0 r_0^n.
$$

 Next, let us take $n_0$ large enough so that $c_0 r_0^{n_0}\leq\frac12$. Let $X_k=X_k(\om)=(\xi_{kn_0},...,\xi_{kn_0+n_0-1})$ and set 
$$
Q_n=Q_n(\om)=\sum_{k=0}^{[n/n_0]-1}\bbI(X_k=(1,1,...,1)).
$$
Then by the mean ergodic theorem $Q_n/n\to\bbP(X_1=(1,1,...,1))=p_0>0$ (almost surely). Thus, for almost every realization of the random dynamical system we get that the   operators on the right hand side of  \eqref{EE2} are composed of $k$ blocks with $\|\cdot\|_{(t)}$-norms of order $O(1)$, at least $[c_1 n]$
 blocks $c_1=\frac14a_0>0$ with $\|\cdot\|_{(t)}$-norms less or equal to $\frac 12$ and the rest of the blocks have $\|\cdot\|_{(t)}$-norms less or equal to $1$. Starting from \eqref{EE1}, using \eqref{EE2} and the above block decomposition, we see that $\bbP$-a.s. for all $n$ large enough and $|t|\geq K_0$ we have
$$
\left|\frac{d^k}{dt^k}\bbE[e^{it S_n^\om}]\right|\leq C(1+|t|)n^k \left(\frac12\right)^{c_1 n}.
$$ 
Hence Assumption \ref{DerAss} is in force for all $m$.

Finally, let us note that the same argument works for a  random dynamical systems $T_{\om}$ and a random function $f_\om$ so that  $\bbP(A)>0$, $A=\{\om: T_{\te^j\om}=T_1, f_{\te^j\om}=f_1, \,\forall \,0\leq j\leq n_0\}$ 
since then we can consider the number of visiting times to $A$.
\end{itemize} 
\end{example}

\subsection{Sequential dynamical systems in a neighborhood of a single system}\label{SDS sec}
Let $(\cX_j,\cB_j,\mu_j)$ be a sequence of probability  space, and let $T_j:\cX_j\to\cX_{j+1}$ be a sequence of measurable maps. Let $\cL_j$ an operator which is defined by the duality relation:
$$
\int_{\cX_j}g\cdot (f\circ T_j)d\mu_{j}=\int (\cL_jg)\cdot fd\mu_{j+1}
$$    
for all bounded functions $g:\cX_{j}\to\bbR$ and $f:\cX_{j+1}\to\bbR$. Let $B_j$ be a Banach space of measurable functions on $\cX_j$, equipped with a norm
$\|\cdot\|_j$ so that $\sup|g|\leq C\|g\|_j$ for some constant $C$ which does not depend on $j$. We assume here that there are constants $A$ and $\del\in(0,1)$ and strictly positive functions $h_j\in B_j$ so that for every $g\in B_j$ and all $n\geq1$ we have
$$
\|\cL_{j+n-1}\circ\cdots\circ\cL_{j+1}\circ\cL_j g-\mu_j(g)h_{j+n}\|_{j+n}\leq A\|g\|_j\del^n.
$$
Let take now a sequence of real-valued functions $f_j\in B_j$ so that $\|f\|:=\sup_j\|f\|_j<\infty$, $\mu_j(f_j)=0$ and define 
$$
S_n=\sum_{j=0}^{n-1}f_j\circ T_{j-1}\circ\cdots\circ T_1\circ T_0(x_0)
$$
where $x_0$ is distributed according to $\mu_0$. Let us define $\cL_j^{(z)}(g)=\cL_j(e^{zf_j})$, where $z\in\bbC$. We further assume here that the map $z\to\cL_j^{(z)}$ is analytic around the complex origin (with values in the space of bounded linear operators between $B_j$ and $B_{j+1}$), uniformly in $j$.   
\begin{example}
\,
\begin{itemize}
\item[(i)] These conditions hold true in the setup of \cite{HaNonl}, where the maps $T_j$ are locally expanding, the space $B_j$ is the space of H\"older continuous functions and $\mu_j$ are measures so that $(T_j)_*\mu_j=\mu_{j+1}$. Note that when the maps are absolutely continuous with respect to some reference measure (e.g. a volume measure) then $\mu_j$ is equivalent to that measure. 
\vskip0.1cm
\item[(ii)] 
These conditions hold true in the setup of \cite{ConzeR}, where $\cX_j=\cX$ and $\mu_j=\mu$ coincide with the same manifold and (normalized) volume measure, respectively, each $T_j$ is a locally expanding map and $B_j=B$ is the space of functions with bounded variation.
\end{itemize}
\end{example}

Next, by applying  a sequential perturbation theorem  \cite[Theorem D.2]{DolgHaf PTRF 2}
  we see that there is a constant $r_0$ so that for every complex parameter $z$ with $|z|\leq r_0$ there are  analytic in $z$ triplets $(\la_j(z),h_j^{(z)},\nu_j^{(z)})$ consisting of a complex non-zero random variable $\la_j(z)$, a complex-valued function $h_j^{(z)}\in B_j$ and a complex continuous linear functional $\nu_j^{(z)}\in B_j^*$ so that $\la_j(0)=1$, $h_j^{(0)}=h_j$ and $\nu_j^{(0)}=\mu_j$, $\nu_j^{(z)}(h_j^{(z)})=\nu_j^{(z)}(\textbf{1})=1$. Moreover, there are constants $A_1>0$ and $\del_1\in(0,1)$ so that for all $j,n$ and a function $g\in 	B_j$ we have
$$
\left\|\frac{\cL_{j+n-1}\circ\cdots\circ\cL_{j+1}\circ\cL_j g}{\la_{j+n-1}(z)\cdots \la_{j+1}(z)\la_j(z)}-\nu_j^{(z)}(g)h_{j+n}^{(z)}\right\|\leq A_1\|g\|\del_1^n.
$$ 
The following result follows exactly like the corresponding result about random dynamical systems in the previous section.

\begin{theorem}
Suppose that $\text{Var}_{\mu_0}(S_n)\geq c_0 n$ for some $c_0$ and all $n$ large enough. Then Assumption \ref{GrowAssum} is in force, and so Theorem \ref{BE1} and Corollaries  \ref{Cor} and \ref{ASIP} hold true.
\end{theorem}

\subsubsection{Linearly growing variances: the perturbative approach}
We refer to \cite[Section 5.3]{NonU BE} for an explanation how to show that $\text{Var}_{\mu_0}(S_n)\geq c_0 n$ when all the maps $T_j$ are close.

\subsubsection{Products of random non-stationary matrices}
Let us take an independent sequence $g_1,g_2,...$ uniformly bounded of random matrices which are not necessarily uniformly distributed. Let us denote that law of $g_i$ by $\mu_i$. 
Let $\mu$ be a  probability distribution $\mu$ with the properties described at the beginning of Section \ref{SecMat} (so that Assumptions  \ref{Ass1 mat} and \ref{Ass2 mat} are in force and Proposition \ref{MatProp} holds true). Let us further assume that $S(\mu)$ is bounded.
Set
$$
\ve=\sup_i\|\mu_i-\mu\|_{TV}.
$$
Then as explained in \cite[Section 5.3.2]{NonU BE}, Assumption \ref{GrowAssum} is in force if $\varepsilon$ is small enough, and so Theorem \ref{BE1} holds.

\section{Additional examples} \label{Add}
Let us note that Assumption \ref{GrowAssum} holds true for every $m$ when 
$$
\gamma_j(S_n/\sig_n)\leq C^j(j!)\sig_n^{-(j-2)}
$$
 for some $C\geq 1$ and all $j\geq3$. Indeed, this condition insures that the Taylor series of the function $\Lambda_n(t)$ and all of its derivatives converge in a neighborhood of the origin (which may depend on the order of differentiation).

\begin{example}
\,
\begin{itemize}
\item Nondegenerate U-Statistics, Characteristic Polynomials in the Circular Ensembles and Determinantal Point Processes as in \cite[Section 3-5]{Dor}.

\item Partial sums $S_n$ of exponentially fast  $\phi$ mixing  uniformly bounded Markov chains,  with linearly fast growing variances (see \cite[Theorem 4.26]{SaulStat}).
\end{itemize}

\end{example}


\begin{thebibliography}{Bow75}
\bibliographystyle{alpha}
\itemsep=\smallskipamount




\bibitem{[3]}
 Agnew, R.P.: Asymptotic expansions in global central limit theorems. Ann. Math. Stat. 30, 721–737
(1959)

\bibitem{ANM}
R. Aimino, M. Nicol and S. Vaienti, {\em Annealed and quenched limit theorems for random expanding dynamical systems}, Probab. Th. Rel. Fields 162, 233-274, (2015).
 
 \bibitem{[1]}
E. Al\'os. A generalization of the Hull and White formula with applications to option pricing approximation. Finance Stoch., 10(3) (2006),
353–365.


\bibitem{AP}
J. Angst, G. Poly, {\em A weak Cram\'er condition and application to Edgeworth expansions}, Electron. J. Probab. \textbf{22} (2017) \# 59, 1-24.

\bibitem{Aus}
M. Austern, T. Liu, {\em Wasserstein-$p$ Bounds in the Central Limit Theorem under Weak Dependence}, preprint, https://arxiv.org/abs/2209.09377v1



\bibitem{Bar}
A.D. Barbour, {\em Asymptotic expansions based on smooth functions in the central limit theorem},
Probab. Th. Rel. Fields {\bf 72} (1986) 289--303. 







\bibitem{BR BE}
P.Beckedorf, A.Rohde,
{\em Non-uniform bounds and Edgeworth expansions in self-normalized limit theorems}, preprint, https://arxiv.org/abs/2207.14402v2


\bibitem{BGVZ2}
V. Bentkus, F.
Gotze, W. Van Zwet  
{\em An Edgeworth Expansion for Symmetric Statistics}. Ann. Stat. {\bf 25} (1997)
 851--896.






\bibitem{Berry}
W. Berry, {\em The accuracy of the Gaussian approximation to the sum of independent variates},
Trans. AMS {\bf 49} (1941) 122--136.


\bibitem{BR76} 
R.N. Bhattacharya, R. Ranga Rao
{\em Normal Approximation and Asymptotic Expansions,} Wiley, New
York-London-Sydney-Toronto (1976) xiv+274 pp.

\bibitem{BGVZ1}
P. Bickel P., F. Gotze, W.  Van Zwet
{\em The Edgeworth Expansion for U-Statistics of Degree Two,}  Ann. Stat. {\bf 14} (1986) 1463--1484.


\bibitem{Bobkov2016}
Bobkov, S.G. {\em Closeness of probability distributions in terms of Fourier–Stieltjes transforms} (English
translation: Russian Math. Surveys). (Russian) Uspekhi Matem. Nauk 71(6), 37–98 (2016)



\bibitem{Bobkov2018}
S.G. Bobkov, {Berry–Esseen bounds and Edgeworth expansions in the central limit theorem for transport distances}, Probab. Theory Relat. Fields (2018) 170:229–262.




\bibitem{BoLac}
 Bougerol, Ph., Lacroix, J. {\em Products of Random Matrices with Applications to Schrödinger Operators}.
Birkhäuser, Boston, Basel, Stuttgart (1985) 
 


\bibitem{Br}
E. Breuillard, {\em Distributions diophantiennes et theoreme limite local sur $\bbR^d$}, Probab.
Th. Rel. Fields {\bf 132} (2005) 39--73.


\bibitem{BP}
Butterley, O., Peyman, E. {\em Exponential mixing for skew products with discontinuities}. Trans. Am. Math. Soc. 369(2), 783--803 (2017).


\bibitem{CJ}
H. Callaert, P. Janssen,
{\em The Berry-Esseen theorem for U statistics}, Ann. Stat. {\bf 6} (1978) 417--421.

\bibitem{CJV}
H. Callaert, P. Janssen, N. Veraverbeke {\em An Edgeworth Expansion for U-Statistics}. Ann. Stat.  {\bf 8} (1980) 299--312. 





\bibitem{CP}
Z. Coelho, W. Parry {\em Central limit asymptotics for shifts of finite type}, Israel J. Math.
{\bf 69} (1990) 235--249.


  \bibitem{Bowen}
R. Bowen, {\em Equilibrium states and the ergodic theory of Anosov diffeomorphisms},
 Lecture Notes in Mathematics, volume 470, Springer Verlag, 1975.



\bibitem{Cheb}
P. Chebyshev. Sur deux th\'eor`emes relatifs aux probabilit\'es. Acta Math.,
14(1) (1890), 305–315.




\bibitem{ConzeR}
J. P. Conze and A. Raugi, {\em Limit theorems for sequential expanding dynamical systems on [0, 1], In:
Ergodic Theory and Related Fields}., Contemp. Math. 430 (2007), 89–121.



\bibitem{Cr28}
H. Cramer {\em On the composition of elementary errors,} 
Skand. Aktuarietid skr. {\bf 1} (1928) 13--74; 141--180.















\bibitem{DolgIsr}
D. Dolgopyat, {\em On mixing properties of compact group extensions of hyperbolic systems},  Isr. J. Math. 130, 157–205 (2002).








\bibitem{DF} D. Dolgopyat, K. Fernando
{\em An error term in the Central Limit Theorem for sums of discrete random variables,} Inter. Res. Math. Not. Issue 21, November 2023, Pages 18664--18713.

\bibitem{DH} D. Dolgopyat, Y. Hafouta
{\em Edgeworth expansion for indepenent bounded integer valued random variables}, Stoch. Proc. App. 152 (2022) Volume 152 486--532.











\bibitem{DolgHaf}
D. Dolgopyat, Y. Hafouta, {\em A Berry-Esseen theorem and Edgeworth expansions for uniformly elliptic inhomogeneous Markov chains}, Probab. Theory Relat. Fields (2022). https://doi.org/10.1007/s00440-022-01177-2


\bibitem{DolgHaf PTRF 2}
D. Dolgopyat, Y. Hafouta, {\em Berry-Esseen theorems for sequences of expanding maps}, Probab. Theory Relat. Fields (2025). https://doi.org/10.1007/s00440-025-01368-7


\bibitem{DolgHaf LLT}
D. Dolgopyat, Y. Hafouta, {\em Local limit theorems for sequences of expanding maps}, https://arxiv.org/abs/2407.08690 


\bibitem{DS}
D. Dolgopyat, O. Sarig, {\em Local limit theorems for inhomogeneous Markov chains},
arXiv:2109.05560.


\bibitem{Dor}
H. D\"oring, P. Eichelsbacher, {\em Moderate deviations via cumulants}, 
J. Theor. Probab.  {\bf 26} (2013) 360--385.











\bibitem{DavorCMP}
D.~{Dragi\v{c}evi\'c}, G.~Froyland, C.~Gonz\'alez-Tokman, and S.~Vaienti.
\newblock {A Spectral Approach for Quenched Limit Theorems for Random Expanding Dynamical Systems},
\newblock {\em Comm. Math. Phys.} 360 (2018), 1121--1187.

\bibitem{DavorTAMS}
D.~{Dragi\v{c}evi\'c}, G.~Froyland, C.~Gonz\'alez-Tokman, and S.~Vaienti.
\newblock {A Spectral Approach for Quenched Limit Theorems for Random Expanding Dynamical Systems},
\newblock {\em Tran. Amer. Math. Soc.} 360 (2018), 1121--1187.

\bibitem{DDH}
D. Dragi\v cevi\' c, Y. Hafouta, 
{\em 	Limit theorems for random expanding or Anosov dynamical systems and vector valued observables}, Ann. Henri Poincar\'e {\bf 21} (2020) 3869--3917.


 \bibitem{EdgeOr}
F.Y. Edgeworth. The asymmetrical probability curve. Proceedings of
the Royal Society of London, 56(336-339) (1894), 271–272

\bibitem{Efron 1976}
B. Efron, ‘Bootstrap methods: Another look at the jackknife’, Ann.
Statist. 7(1) (1979), 1–26.


\bibitem{[27]}
B. Efron. The jackknife, the bootstrap and other resampling plans, volume 38 of CBMS-NSF Regional Conference Series in Applied Mathematics. Society for Industrial and Applied Mathematics (SIAM),
Philadelphia, Pa., 1982.



\bibitem{Ess}
C.-G. Esseen {\em Fourier analysis of distribution functions. A mathematical study of the Laplace-Gaussian law,} 
Acta Math. {\bf 77} (1945) 1--125.




\bibitem{Esseen1956}
C.-G. Esseen, {\em A moment inequality with an application to the central limit theorem}, Skand. Aktuarietidskr. 39, (1956) 160--170.



\bibitem{Feller}
W. Feller, {\em An introduction to probability theory and its applications}, Vol. II., 2d edition, John Wiley \& Sons, Inc., New York-London-Sydney, 1971.


\bibitem{FL}
K. Fernando, C. Liverani {\em Edgeworth expansions for weakly dependent random variables}, 
Ann. de l'Institut Henri Poincare Prob. \& Stat. {\bf 57}  (2021) 469--505.
 
\bibitem{FP}
K. Fernando, F. P\`ene {\em Expansions in the local and the central limit theorems for dynamical systems}, 
arXiv:2008.08726. 


\bibitem{[33]}
J. Fouque, G. Papanicolaou, R. Sircar, and K. Solna. {\em Singular perturbations in option pricing}. SIAM Journal on Applied Mathematics,
63(5) (2003), 1648–1665










\bibitem{GO}
S.~Gou{\"e}zel.
{\em Berry--Esseen theorem and local limit theorem for non uniformly
  expanding maps},
 Annales IHP Prob. \& Stat. 
  {\bf 41} (2005),  997--1024.



\bibitem{GH}
Y.~Guivarc'h, J.~Hardy.
{\em Th{\'e}or{\`e}mes limites pour une classe de cha{\^\i}nes de markov
  et applications aux diff{\'e}omorphismes d'anosov,}
 Annales IHP Prob. \& Stat. 
 {\bf 24} (1988), 73--98.


\bibitem{Guiv}
Y. Guivar'ch {\em Spectral gap properties and limit theorems for some
random walks and dynamical systems}, Proc. Symp. Pure Math. 89, 279–310 (2015).







\bibitem{HK}
Y. Hafouta and Yu. Kifer, {\em Nonconventional limit theorems and random dynamics}, 
World Scientific, Singapore, 2018.
  
 
\bibitem{Ha}
Y. Hafouta, {\em On the asymptotic moments and Edgeworth expansions of some processes in random dynamical environment},  J. Stat Phys {\bf 179} (2020) 945--971. 
  
\bibitem{HaNonl}
Y. Hafouta, {\em Limit theorems for some time dependent expanding dynamical systems}, Nonlinearity 
{\bf 33} (2020)  6421--6460. 


  \bibitem{HafYT}
Y. Hafouta, {\em Limit Theorems for Random Non-uniformly Expanding or Hyperbolic Maps with Exponential Tails},  Ann. Henri Poincar\'e  23, 293–332 (2022).


\bibitem{HafASIP}
Y. Hafouta, {\em An almost sure invariance principle for some classes of non-stationary mixing sequences}, Statistics and Probability Letters,
Volume 193 (2023),
109728

\bibitem{RPF2022}
Y. Hafouta, {\em Explicit conditions for the CLT and related results for non-uniformly partially expanding random dynamical systems via effective RPF rates}, Advances in Mathematics
Volume 426, 1 August 2023, 109109

\bibitem{MarShif1}
Y. Hafouta, {\em Statistical properties of Markov shifts (part I)}, https://arxiv.org/abs/2510.07757.

\bibitem{NonU BE}
Y. Hafouta, {\em Non-uniform Berry-Esseen theorems for weakly dependent random variables}, arXiv:2210.07204v4











\bibitem{[39]}
P. Hall. The bootstrap and Edgeworth expansion. Springer Series in
Statistics. Springer-Verlag, New York, 1992.



\bibitem{Hall}
P. Hall, {\em Edgeworth expansion for Student's t statistic under minimal moment conditions},
 Ann. Probab. {\bf 15} (1987) 920--931.


\bibitem{HH}
H.~Hennion, L.~Herv{\'e}.
 {\em Limit theorems for {M}arkov chains and stochastic properties of
  dynamical systems by quasi-compactness}, 
  Springer  Lecture Notes in
  Math. {\bf 1766} (2001).


\bibitem{HP}
L.~Herv{\'e}, F. P\`ene, 
{\em The Nagaev--Guivarch method via the Keller-Liverani theorem,} 
Bull. Soc. Math. France {\bf 138} (2010) 415--489.


\bibitem{IL} I. A. Ibragimov, Yu. V. Linnik
{\em Independent and stationary sequences of random variables,} 
Wolters-Noordhoff Publishing, Groningen, 1971. 443 pp.



\bibitem{Jir0} M. Jirak
{\em Berry-Esseen theorems under weak dependence}, Ann. Prob. {\bf 44} (2016)
 2024--2063.





 \bibitem{KelLiv}
G. Keller, C. Liverani, {\em A Spectral Gap for a One-dimensional Lattice of Coupled Piecewise Expanding Interval Maps},   In: Dynamics of Coupled Map Lattices and of Related Spatially Extended Systems. Lecture Notes in Physics, vol 671. Springer, Berlin, Heidelberg (2005).









\bibitem{[51]}
S.N. Lahiri, Resampling methods for dependent data (Springer Series
in Statistics. Springer-Verlag, New York, 2003).




\bibitem{38.}
Li, J.: {\em Fourier decay, Renewal theorem and Spectral gaps for random walks on split semisimple Lie
groups}, arXiv:1811.06484.


\bibitem{MelNic}
I. Melbourne and M. Nicol,  {\em Almost Sure Invariance Principle for Nonuniformly Hyperbolic Systems}. Commun. Math. Phys. 260, 131-146 (2005). 

\bibitem{Nag1}
S.V. Nagaev, {\em Some limit theorems for stationary Markov chains},
Theory Probab. Appl. {\bf 2} (1957) 378-406.

\bibitem{Nag2}
S.V. Nagaev, {\em More exact statements of limit theorems for homogeneous Markov chains}, 
Theory Probab. Appl. {\bf 6} (1961), 62-81.

\bibitem{Nag65}
S.V. Nagaev, {\em Some limit theorems for large deviations}. Theory Probab. Appl. 10(1), 214–235 (1965)

\bibitem{Osipov67}
L.V Osipov, {\em Asymptotic expansions in the central limit theorem}. (Russian) Vestnik Leningrad. Univ.
19, 45–62 (1967)

\bibitem{Osipov72}
L.V Osipov,  {\em On asymptotic expansions of the distribution function of a sum of random variables with
non uniform estimates for the remainder term}. (Russian) Vestnik Leningrad. Univ. 1, 51–59 (1972)


\bibitem{PelPTRF}
M. Peligrad, {Central limit theorem for triangular arrays of non-homogeneous Markov chains},
Probab. Theory Relat. Fields (2012) 154:409--428. 










\bibitem{Petrov72}
V.V. Petrov, {\em Sums of independent random variables}. Translated from the Russian by A. A. Brown.
Ergebnisse der Mathematik und ihrer Grenzgebiete, Band 82. Springer, New York-Heidelberg, 1975.
x+346 pp. Russian ed.: Moscow, Nauka (1972)




\bibitem{RioBE}
E. Rio, {\em Sur le th\' eor\` eme de Berry-Esseen pour les suites faiblement
d\' pendantes}, Probab. Th. Relat. Fields 104 (1996), 255-282


\bibitem{Rio2009}
E. Rio, {\em Upper bounds for minimal distances in the central limit theorem}, 
Annales de l’Institut Henri Poincaré - Probabilités et Statistiques
2009, Vol. 45, No. 3, 802--817.













\bibitem{RinRot}
Y. Rinott and V. Rotar {\em On Edgeworth expansions for dependency-neighborhoods chain structures and Stein's method,} Probab. Theory Rel. Fields {\bf 126}  (2003) 528--570. 


\bibitem{RE}
J.~Rousseau-Egele.{\em 
 Un th{\'e}oreme de la limite locale pour une classe de
  transformations dilatantes et monotones par morceaux,}
Ann. Prob. {\bf 11} (1983) 772--788.







\bibitem{SaulStat}
L. Saulis, V.A. Statulevicius, {\em Limit Theorems for Large Deviations}, 
Kluwer, Dordrecht, Boston, 1991.









\bibitem{Stein72}
C. Stein, {\em A bound for the error in the normal approximation to the distribution of a sum of dependent random variables}, Berkeley Symposium on Mathematical Statistics and Probability, 1972: 583-602 (1972)









\bibitem{Varandas}
M. Stadlbauer, S. Suzuki, P. Varandas {\em Thermodynamic formalism for random non-uniformly expanding maps}, Commun. Math. Phys. 385, 369–427 (2021)





\bibitem{Y1}
L.S. Young, {\em Statistical properties of dynamical systems with some hyperbolicity}, 
 Ann. Math. 7 (1998) 585-650.

 




\end{thebibliography}
\end{document}